\documentclass[reqno,11pt]{amsart}
\usepackage{amsmath, amsfonts,amsthm,amssymb,amsbsy,upref,color,graphicx,amscd,hyperref,enumerate,enumitem,relsize}
\usepackage{color}
\usepackage[active]{srcltx}
\usepackage[latin1]{inputenc}
\usepackage{bbm}
\usepackage{graphicx}
\usepackage{subfig} 
\usepackage{xfrac}
\usepackage{comment} 

\usepackage{dsfont}

\def\eps{{\varepsilon}}
\def\N{\mathbb{N}}
\def\O{\Omega}

\def\R{\mathbb{R}}

\newcommand{\be}{\begin{equation}}
\newcommand{\ee}{\end{equation}}

\newcommand{\ind}{\mathbbm{1}}

\theoremstyle{plain}
\newtheorem{teo}{Theorem}[section]
\newtheorem{lm}[teo]{Lemma}
\newtheorem{prop}[teo]{Proposition}
\newtheorem{coro}[teo]{Corollary}

\theoremstyle{definition}
\newtheorem{definition}[teo]{Definition}

\newtheorem{oss}[teo]{Remark}

\def\Xint#1{\mathchoice
   {\XXint\displaystyle\textstyle{#1}}%
   {\XXint\textstyle\scriptstyle{#1}}%
   {\XXint\scriptstyle\scriptscriptstyle{#1}}%
   {\XXint\scriptscriptstyle\scriptscriptstyle{#1}}%
   \!\int}
\def\XXint#1#2#3{{\setbox0=\hbox{$#1{#2#3}{\int}$}
     \vcenter{\hbox{$#2#3$}}\kern-.5\wd0}}

\def\aver#1{\Xint-_{#1}}

\DeclareMathOperator{\dive}{div}

\usepackage{refcount}

\newcounter{cte}

\numberwithin{equation}{section}

\begin{document}

\title[Regularity of optimal sets]{Regularity of optimal sets for some functional involving eigenvalues of an operator in divergence form}

\begin{abstract}
	In this paper we consider minimizers of the functional 
	\begin{equation*}
	\min \big\{ \lambda_1(\O)+\cdots+\lambda_k(\O) + \Lambda|\O|, \ : \ \O \subset D \text{ open} \big\}
	\end{equation*}
	where $D\subset\R^d$ is a bounded open set and where $0<\lambda_1(\O)\leq\cdots\leq\lambda_k(\O)$ are the first $k$ eigenvalues on $\O$ of an operator in divergence form with Dirichlet boundary condition and with H\"{o}lder continuous coefficients. We prove that the optimal sets $\O^\ast$ have finite perimeter and that their free boundary $\partial\O^\ast\cap D$ is composed of a {\it regular part}, which is locally the graph of a $C^{1,\alpha}$-regular function, and a {\it singular part}, which is empty if $d<d^\ast$, discrete if $d=d^\ast$ and of Hausdorff dimension at most $d-d^\ast$ if $d>d^\ast$, for some $d^\ast\in\{5,6,7\}$.
\end{abstract}

\author{Baptiste Trey}
\date{\today}
\maketitle
\tableofcontents

\section{Introduction}

This paper is dedicated to the regularity properties of the minimizers to the problem
\begin{equation}\label{e:shapeopt}
\min \big\{ \lambda_1(\O)+\cdots+\lambda_k(\O) + \Lambda|\O| \ : \ \O \subset D \text{ open} \big\}
\end{equation}
where $D\subset\R^d$ is a bounded open set (a box), $\Lambda$ is a positive constant and $0<\lambda_1(\O)\leq\cdots\leq\lambda_k(\O)$ stand for the first $k$ eigenvalues (counted with the due multiplicity) of an operator in divergence form. More precisely, we consider the operator $-b(x)^{-1}\dive(A_x\nabla\cdot)$, where the matrix-valued function $A:D\rightarrow \text{Sym}_d^+$ is uniformly elliptic with H\"{o}lder continuous coefficients, and $b\in W^{1,\infty}(D)$ is a positive Lipschitz continuous function bounded away from $0$. This means that for every eigenvalue $\lambda_i(\Omega)$ there exists an eigenfunction $u_i\in H^1_0(\Omega)$ such that 
\begin{equation}\label{e:def_op}
\left\{
      \begin{aligned}
        -\dive(A\nabla u_i) &= \lambda_i(\Omega) \,b\,u_i & &\text{ in}\quad \O \\
        u_i&=0 & &\text{ on}\quad \partial\O. \\
      \end{aligned}
    \right.
\end{equation}

We now state in the following theorem the main result of this present paper.

\begin{teo}\label{t:main}
Let $D\subset\R^d$ be a bounded open set and let $A:D\rightarrow\text{Sym}_d^+$, $b\in W^{1,\infty}(D)$ satisfying \eqref{e:holderA}, \eqref{e:ellipA} and \eqref{e:hypfctb} (see below). Then every solution $\O^\ast$ to the problem \eqref{e:shapeopt} has finite perimeter. Moreover, the free boundary $\partial\Omega^\ast\cap D$ can be decomposed into the disjoint union of a regular part $\text{Reg}(\partial\O^\ast\cap D)$ and a singular part $\text{Sing}(\partial\O^\ast\cap D)$, where: 
\begin{enumerate}
\item $\text{Reg}(\partial\O^\ast\cap D)$ is locally the graph of a $C^{1,\alpha}$-regular function. \\
If, moreover, $a_{i,j}\in C^{k,\delta}(D)$ and $b\in C^{k-1,\delta}(D)$ for some $\delta\in(0,1)$ and $k\geq 1$, then $\text{Reg}(\partial\O^\ast\cap D)$ is locally the graph of a $C^{k+1,\alpha}$-regular function. 
\item for a universal constant $d^\ast\in\{5,6,7\}$ (see Definition \ref{d:dstar}), $\text{Sing}(\partial\O^\ast\cap D)$ is: 
\begin{itemize}
\item empty if $d<d^*$;
\item discrete if $d=d^*$; 
\item of Hausdorff dimension at most $(d-d^\ast)$ if $d>d^*$.
\end{itemize}
\end{enumerate}
\end{teo}

The problem \eqref{e:shapeopt} can also be considered in the class of the quasi-open sets, but we stress out that it is the same thing. Indeed, preliminary results, inspired by the work of David and Toro in \cite{david-toro-15} (see also \cite{david-engelstein-garcia-toro-19}), have already been obtained in \cite{trey-19} in view to prove the regularity of the minimizers to \eqref{e:shapeopt}. The main results of the paper are stated in theorem \ref{t:lipquasimin}, where the author shows that if a quasi-open set $\O^\ast$ is solution, among the class of quasi-open sets, to the problem \eqref{e:shapeopt}, then the first $k$ eigenfunctions on $\O^\ast$ are locally Lipschitz continuous, and hence $\O^\ast$ is an open set.

One of the main interest and difficulty of this paper is to consider an operator with variable coefficients. This case is more involved than the case of the Laplacian and has been studied only recently. We notice that our result is quite general and applies, for instance, to an operator with drift $-\Delta+\nabla\Phi\cdot\nabla$ or in the case of a manifold.

The first result concerning the regularity of the free boundary of optimal sets (for spectral functionals) was established by Brian\c con and Lamboley in \cite{briancon-lamboley-09}, where they consider the minimization problem of the first eigenvalue of the Dirichlet Laplacian with inclusion and volume constraints. More precisely, using the strategy developed by Alt and Caffarelli in \cite{alt-caffarelli-81}, they prove that the optimal sets for the problem
\begin{equation}\label{e:min_lambda1}
\min \big\{ \lambda_1(\O) \ : \ \O \subset D \text{ open, }|\O|\leq m \big\}
\end{equation}
have $C^\infty$-regular boundary (inside $D$) up to a singular set whose $(d-1)$-Hausdorff measure is zero (provided that the box is bounded and connected). In \cite{mazzoleni-terracini-velichkov-17}, Mazzoleni, Terracini and Velichkov study the regularity properties of sets that minimize the sum of the first $k$ eigenvalues of the Dirichlet Laplacian among all sets of fixed volume, that is, minimizers of
\begin{equation}\label{e:min_sum_lambda}
\min \big\{ \lambda_1(\O)+\cdots+\lambda_k(\O) \ : \ \O \subset \R^d \text{ open},\ |\O|=1 \big\}.
\end{equation}
They prove that the regular part of the boundary of an optimal set is $C^\infty$-regular and, thanks to a dimension's reduction argument due to Weiss (see \cite{weiss-99}), that the singular set is of dimension at most $d-d^\ast$, hence improving the smallness estimate of the singular set.
Meanwhile, Kriventsov and Lin consider in \cite{kriventsov-lin-18} a more general functional and prove that minimizers of 
\begin{equation*}
\min \big\{ F(\lambda_1(\O),\cdots,\lambda_k(\O))+|\O| \ : \ \O \subset \R^d \text{ open} \big\}.
\end{equation*}
are $C^\infty$-regular up to a singular set of dimension at most $d-3$. Here, $F:\R^k\rightarrow\R$ is a function of class $C^1$ which is strictly increasing in each variable ($\partial_iF\geq c>0$). In \cite{kriventsov-lin-19}, they also obtain a regularity result in the case where the functional $F$ is non-decreasing in its parameters, which hence apply to minimizers of
\begin{equation*}
\min \big\{ F(\lambda_{k_1}(\O),\cdots,\lambda_{k_n}(\O))+|\O| \ : \ \O \subset \R^d \text{ quasi-open} \big\},
\end{equation*}
where the first eigenvalue is not necessary involved.
Notice that in these problems, the main difficulty is to deal with higher eigenvalues since they have a min-max variational characterization. 

On the other hand, regularity problems involving different operators have been studied only recently. In \cite{russ-trey-velichkov-19}, the authors prove the regularity of the minimizers to \eqref{e:min_lambda1} where $\lambda_1$ now stands for the first eigenvalue of a drifted operator $-\Delta+\nabla\Phi\cdot\nabla$ with Dirichlet boundary condition (for some $\Phi\in W^{1,\infty}(D,\R^d)$), and therefore extend the result of Brian\c con and Lamboley. We highlight that the operator considered in this paper (see \eqref{e:def_op}) is more general than the operator with drift $-\Delta+\nabla\Phi\cdot\nabla$ which corresponds to the special case where $A=e^{-\Phi}\text{Id}$ and $b=e^{-\Phi}$. Recently, Lamboley and Sicbaldi  successfully treated the minimization problem \eqref{e:min_lambda1} in the manifold setting with the Laplace-Beltrami operator (see \cite{lamboley-sicbaldi-19}). They prove the existence of an optimal set among quasi-open set provided that the manifold $M$ is compact and that optimal sets are $C^\infty$-regular if $M$ is connected (and $C^\infty$) up to $(d-d^\ast)$-dimensional singular set. 

Let us also mention that some regularity results have also been established in the context of multiphase shape optimization problems involving eigenvalues (see, for instance, \cite{conti-terracini-verzini-03}, \cite{caffarelli-lin-07}, \cite{ramos-tavares-terracini-16}, \cite{spolaor-trey-velichkov-19})

We notice that we deal with a penalized functional and that it is natural to expect that a similar result also holds with a volume constraint as in \eqref{e:min_lambda1}, but we will not address this question in this paper since our main motivation is to treat the case of an operator with variable coefficients. 

\subsection{Preliminaries and notations}
We will use the following notations throughout this paper. 
We fix a matrix-valued function $A=(a_{ij})_{ij} : D\rightarrow \text{Sym}_d^+$, where $\text{Sym}_d^+$ denotes the family of the real positive symmetric $d \times d$ matrices, which is uniformly elliptic and has H\"{o}lder continuous coefficients. Precisely, there exist positive constants $\delta_{\text{\tiny\sc A}},c_{\text{\tiny\sc A}}>0$ and $\lambda_{\text{\tiny\sc A}}\geq 1$ such that
\begin{equation}\label{e:holderA}
|a_{ij}(x)-a_{ij}(y)|\le c_{\text{\tiny\sc A}}|x-y|^{\delta_{\text{\tiny\sc A}}},\quad\text{for every}\quad i,j\quad\text{and}\quad  x,y\in D\,;
\end{equation}
\begin{equation}\label{e:ellipA}
\frac{1}{\lambda_{\text{\tiny\sc A}}^2}|\xi|^2\le \xi\cdot A_x\,\xi=\sum_{i,j=1}^da_{ij}(x)\xi_i\xi_j\le \lambda_{\text{\tiny\sc A}}^2|\xi|^2,\quad\text{for every}\quad x\in D \quad\text{and}\quad \xi\in\R^d.
\end{equation}    
We also fix a Lipschitz continuous function $b \in W^{1,\infty}(D)$ which we assume to be positive and bounded away from zero: there exists $c_b>0$ such that
\begin{equation}\label{e:hypfctb}
c_b^{-1} \leq b(x) \leq c_b \quad \text{for almost every} \quad x\in D.
\end{equation}

We set $m=b\,dx$ and we define, for any an open set $\O\subset D$, the spaces $L^2(\O;m)=L^2(\O)$ and $H^1_0(\O;m)=H^1_0(\O)$ endowed respectively with the norms
\begin{equation*}
\|u\|_{L^2(\O;m)}=\left(\int_{\O}u^2\,dm\right)^{1/2}\qquad\text{and}\qquad \|u\|_{H^1(\O;m)}=\|u\|_{L^2(\O;m)} + \|\nabla u\|_{L^2(\O)}.
\end{equation*}
By the Lax-Milgram theorem and the Poincar\'{e} inequality, for every $f \in L^2(\O,m)$ there exists a unique solution $u \in H^1_0(\O,m)$ to the problem
\[ -\dive(A\nabla u) = fb \text{ in } \O, \quad u \in H^1_0(\O,m). \]
The resolvent operator $R_{\O} : f \in L^2(\O;m)\rightarrow H^1_0(\O;m) \subset L^2(\O;m)$ defined as $R_{\O}(f)=u$ is continuous, self-adjoint, positive and compact (since $H^1_0(\O;m)$ is compactly embedded into $L^2(\O;m)$, because $b\geq c_b>0$). Therefore, the operator $-b^{-1}\dive(A\nabla\cdot)$ in $\O$ has a discrete spectrum which consists in real and positive eigenvalues denoted by 
\[ 0< \lambda_1(\O) \leq \lambda_2(\O) \leq \cdots \leq \lambda_k(\O) \leq \cdots \]
For every $\lambda_i(\O)$ there exists an eigenfunction $u_i\in H^1_0(\O;m)$ satisfying
\begin{equation*}
-\dive(A\nabla u_i)=\lambda_i(\O)\,b\,u_i\quad\text{in}\quad \O,
\end{equation*}
where the PDE is intended in the weak sense, that is
\begin{equation*}
\int_{\O}A\nabla u_i\cdot\nabla\varphi\,dx=\lambda_i(\O)\int_{\O}u_i\varphi\,dm\qquad\text{for every}\quad\varphi\in H^1_0(\O).
\end{equation*}
Moreover, the eigenfunctions $(u_i)_{i\in\N}$ (on an open set $\O\subset D$) will always be normalized with respect to the norm $\|\cdot\|_{L^2(\O;m)}$ and form an orthonormal system in $L^2(\O;m)$, that is
\begin{equation*} \int_{\O}u_iu_j\,dm=\delta_{ij}:=\left\{
\begin{aligned}
1 &\quad\text{if} &i=j, \\
0 &\quad\text{if} &i\neq j. \\
\end{aligned}\right. 
\end{equation*}

We denote by $H^1_0(\O,\R^k)$ the space of all vector-valued function $U=(u_1,\dots,u_k):\O\to\R^k$ such that $u_i\in H^1_0(\O)$, endowed with the norm
\begin{equation*}
\|U\|_{H^1(\O)}=\|U\|_{L^2(\O)}+\|\nabla U\|_{L^2(\O)}= \sum_{i=1}^k \big(\|u_i\|_{L^2(\O)} +\|\nabla u_i\|_{L^2(\O)} \big).
\end{equation*}
Similarly, we will also need the following norms for $U=(u_1,\dots,u_k):\O\to\R^k$
\begin{equation*}
\|U\|_{L^1(\O)}=\sum_{i=1}^k\|u_i\|_{L^1(\O)}\qquad\text{and}\qquad \|U\|_{L^\infty(\O)}=\sup_{i=1}^k\|u_i\|_{L^\infty(\O)}.
\end{equation*}
Moreover, for $U=(u_1,\dots,u_k):\O\to\R^k$ we set $|U|=u_1^2+\cdots+u_k^2$, $|\nabla U|^2=|\nabla u_1|^2+\cdots+|\nabla u_k|^2$ and $A\nabla U\cdot\nabla U=A\nabla u_1\cdot\nabla u_1+\cdots+A\nabla u_k\cdot\nabla u_k$. 
Finally, for $f=(f_1,\dots,f_k)\in L^2(\O,\R^k)$ we say that $U=(u_1,\dots,u_k)\in H^1_0(\O,\R^k)$ is solution to the equation 
\begin{equation*}
-\dive(A\nabla U)=f\quad\text{in}\quad\O,\qquad U\in H^1_0(\O,\R^k)
\end{equation*}
if, for every $i=1,\dots,k$, the component $u_i$ is solution to the equation
\begin{equation*}
-\dive(A\nabla u_i)=f_i\quad\text{in}\quad\O,\qquad u_i\in H^1_0(\O).
\end{equation*}

We summarize in the following theorem the main results obtained in \cite{trey-19}.
\begin{teo}\label{t:lipquasimin}
Let $D\subset\R^d$ be a bounded open set and let $A:D\rightarrow\text{Sym}_d^+$, $b\in L^\infty(D)$ satisfying \eqref{e:holderA}, \eqref{e:ellipA} and \eqref{e:hypfctb}. Then the minimum 
\begin{equation}\label{e:shapeopt_qo}
\min \big\{ \lambda_1(\O)+\cdots+\lambda_k(\O) + \Lambda|\O| \ : \ \O \subset D \text{ quasi-open } \big\}
\end{equation}
is achieved. Moreover, the vector $U=(u_1,\dots,u_k)\in H^1_0(\O^\ast,\R^k)$ of the first $k$ normalized eigenfunctions on any optimal set $\O^\ast$ for \eqref{e:shapeopt_qo} satisfies:
\begin{enumerate}
\item $U\in L^\infty(D)$ and is a locally Lipschitz continuous function in $D$. In particular, $\O^\ast$ is an open set.
\item $U$ satisfies the following quasi-minimality property:
for every $C_1>0$ there exist constants $\eps\in (0,1)$ and $C>0$, depending only on $d, k, C_1, \|U\|_{L^\infty}$ and $|D|$, such that
\begin{equation}\label{e:quasiminUa}
\int_DA\nabla U\cdot\nabla U\,dx + \Lambda|\{|U|>0\}| \leq \big(1+C\|U-\tilde{U}\|_{L^1}\big) \int_DA\nabla\tilde{U}\cdot\nabla\tilde{U}\,dx + \Lambda|\{|\tilde{U}|>0\}|,
\end{equation}
for every $\tilde{U} \in H^1_0(D,\R^k)$ such that $\|U-\tilde{U}\|_{L^1}\leq\eps$ and $\|\tilde{U}\|_{L^\infty}\leq C_1$.
\end{enumerate}
\end{teo}

\subsection{General strategy and main points of the proof} 
Throughout this paper we will always denote by $\O^\ast$ an optimal set to the problem \eqref{e:shapeopt}.
In section \ref{s:gen_prop}, we reduce to the case where $A=Id$ and prove that the vector $U=(u_1,\dots,u_k)$ of the first $k$ eigenfunctions on $\O^\ast$ is, in some new set of coordinates, a quasi-minimizer of the Dirichlet energy in small balls centered at the origin (Proposition \ref{p:changevarU}). We notice that we perform a change of coordinates near every point $x\in\partial\O^\ast$ and hence that one of the main issue is to deal with functions $U_x=U\circ F_x$ which depends on the point $x$ (see \eqref{e:def_F} for the definition of $F_x$). We adapt the strategy developed by David and Toro in \cite{david-toro-15} to prove that $U_x$ is non-degenerate (Proposition \ref{p:nondeg}). Using an idea of Kriventsov and Lin in \cite{kriventsov-lin-18}, we show that the first eigenfunction $u_1$ is non degenerate in $\O^\ast_1$ (Proposition \ref{p:nondegu1}), where $\O^\ast_1$ denotes any connected component of $\O^\ast$ where $u_1$ is positive. From this result we then deduce a uniform growth of $u_1$ near the boundary $\partial\O^\ast_1$ and a density estimate for $\O^\ast_1$. 

We notice that, unlike in \cite{mazzoleni-terracini-velichkov-17}, the optimal set $\O^\ast$ may not be connected. Indeed, the geometrical constraint imposed by the box $D$ and the presence of variable coefficients do not allow to translate the connected components of $\O^\ast$ and hence to prove as in \cite{mazzoleni-terracini-velichkov-17} that $\O^\ast$ is connected. However, we prove in Proposition \ref{p:connected_comp_sep} that the connected components of $\O^\ast$ cannot meet inside $D$. Therefore, in order to prove Theorem \ref{t:main} it is enough to prove only the regularity of $\O^\ast_1$
(see also remark \ref{r:redu_Oast1} below). This result comes from the structure of the blow-up limits studied in section \ref{s:blow-up}, where we in particular prove that the blow-up limits are one-homogeneous functions and solution of the Alt-Caffarelli functional.

Section \ref{s:reg} is then dedicated to the regularity of $\O^\ast_1$. Since we work with the first $k$ eigenfunctions in a new set of coordinates, namely with $U_x$, we define the regular part of $\O^\ast_1$ in a different way than in \cite{mazzoleni-terracini-velichkov-17} (see Definition \ref{d:reg_sing}).
Then, we show as in \cite{mazzoleni-terracini-velichkov-17} that we can reduce to a one-phase problem, for which the regularity of the free boundary was proved by De Silva (see \cite{de-silva-11} and \cite[Appendix A]{spolaor-trey-velichkov-19}). To this aim, we prove that $\O^\ast_1$ is a non-tangentially accessible (NTA) domain near the regular points and we prove a boundary Harnack principle for the eigenfunctions $U=(u_1,\dots,u_k)$ on $\O^\ast_1$. More precisely, we prove that for every $x_0$ on the regular part of the boundary $\partial\O^\ast_1$, the limits $g_i(x_0)=\lim_{x\rightarrow x_0}\frac{u_i(x)}{u_1(x)}$ exist and define H\"{o}lder continuous functions $g_i:\partial\O^\ast_1\cap B_r(x_0)\rightarrow\R$.  We notice that one difficulty comes from the presence of the function $b$ and that it is the only point in the paper where the Lipschitz continuity assumption on $b$ is needed.  As a consequence, we deduce that $u_1$ satisfies the following optimality condition 
\begin{equation*}
\big|A_{x}^{\sfrac12}[\nabla u_1(x)]\big|=g(x)\sqrt{\Lambda}\quad\text{for every}\quad x\in\partial\O^\ast_1\cap B_r(x_0),
\end{equation*}
where $g$ is an H\"{o}lder continuous function depending on the functions $g_i$ (see \eqref{e:def_g}).
In subsection \ref{sub:est_dim} we provide an estimation of the singular set by proving that we can apply the strategy developed by Weiss in \cite{weiss-99} to the case of an operator in divergence form (see Lemmas \ref{l:comp_weiss} and \ref{l:conv_sing_sets}).

\begin{oss}[On the connected components of the optimal sets]\label{r:redu_Oast1}
We highlight that it is enough to prove the regularity of any connected component of $\O^\ast$ where the first eigenfunction is positive. Indeed, if $\O^\ast_0$ is a connected component of $\O^\ast$, then there exists $k_0>0$ such that $\lambda_i(\O^\ast_0)\in\{\lambda_1(\O^\ast),\dots,\lambda_k(\O^\ast)\}$ for any $i\in\{1,\dots,k_0\}$ and $\lambda_i(\O^\ast_0)\notin\{\lambda_1(\O^\ast),\dots,\lambda_k(\O^\ast)\}$ for any $i>k_0$. Using that $\sigma(\O^\ast)=\sigma(\O^\ast_0)\cup\sigma(\O^\ast\setminus\O^\ast_0)$, it is straightforward to check that $\O^\ast_0$ is solution to the problem \eqref{e:shapeopt} with $k=k_0$ and $D=D\setminus(\overline{\O^\ast\setminus\O^\ast_0})$. Notice also that the connected components of $\O^\ast$ cannot meet inside $D$ (see Proposition \ref{p:connected_comp_sep}). 

Moreover, we notice that $\O^\ast$ has at most $k$ connected components. Indeed, denote by $\O^\ast_i$ a connected component of $\O^\ast$ such that $\lambda_i(\O^\ast)\in\sigma(\O^\ast_i)$. Then, it turns out that the first $k$ eigenvalues on $\O^\ast$ coincide with the first $k$ eigenvalues on $\cup_{i=1}^k\O^\ast_i$ and therefore we have $|\cup_{i=1}^k\O^\ast_i|=|\O^\ast|$ (since otherwise the optimality of $\O^\ast$ gives a contradiction).
\end{oss}

\section{General properties}\label{s:gen_prop}
In this section we study some properties of the optimal sets $\O^\ast$ to the problem \eqref{e:shapeopt} and of its first normalized eigenfunctions $U=(u_1,\dots,u_k)$. We first prove that the optimal sets have finite perimeter and that the vector $U$ is non degenerate. We then prove that the first eigenfunction $u_1$ is non degenerate on any connected component $\O^\ast_1$ of $\O^\ast$ where $u_1$ is positive. As a consequence, we show that $\O^\ast_1$ satisfies a density estimate. We conclude the section with an almost Weiss type formula for $U$.

\subsection{Finiteness of the perimeter}
We prove that the De Giorgi perimeter of any optimal set to the problem \eqref{e:shapeopt} is finite. We follow the strategy introduced by Bucur in \cite{bucur-12} for the eigenvalues of the Dirichlet Laplacian (see also \cite{mazzoleni-terracini-velichkov-18} and \cite{russ-trey-velichkov-19}). Together with a density estimate for the optimal sets $\O^\ast$ (Proposition \ref{p:density_est}), this provides a kind of smallness of the singular set of $\O^\ast$ (see section \ref{sub:est_dim}). The proof of this result will also be used to obtain a non-degeneracy property of the first eigenfunction $u_1$ on $\O^\ast_1$ (Lemma \ref{l:nondegl}).

\begin{prop}\label{p:finitperi}
Let $\O^\ast\subset D$ be an optimal set for the problem \eqref{e:shapeopt}. Then $\O^\ast$ is a set finite perimeter in $\R^d$.
\end{prop}

\begin{proof}
Let $U=(u_1,\dots,u_k)\in H^1_0(\O^\ast,\R^k)$ be the vector of normalized eigenfunctions on $\O^\ast$.
We prove that $\{|u_i|>0\}$ is a set of locally finite perimeter in $D$ for every $i\in\{1,\dots,k\}$. This then implies that the optimal set $\O^\ast=\{|U|>0\}$ has finite perimeter. Let $x\in\partial\{|u_i|>0\}\cap D$ and assume for simplicity that $x=0$. Let $r>0$ be small, $t\in(0,1)$ and $\eta\in\mathcal{C}_c^\infty(B_r)$ be such that $0\leq\eta\leq 1$, $\{\eta=1\}=B_{r/2}$ and $\|\nabla\eta\|_{L^\infty}\leq C/r$. We set
\begin{align*}
    u_{i,t}=\eta(u_i-t)^+ -\eta(u_i+t)^-+(1-\eta)u_i=\left\{
    \begin{array}{ll}
         u_i-t\eta &\quad\text{if}\quad u_i\geq t, \\
         (1-\eta)u_i &\quad\text{if}\quad |u_i|<t, \\
         u_i+t\eta &\quad\text{if}\quad u_i\leq t,
    \end{array} \right.
\end{align*}
and $U_t=(u_1,\dots,u_{i,t},\dots,u_k)\in H^1_0(D,\R^k)$, where $u_{i,t}$ stands at the $i$-th position. Notice that we have $U-U_t\in H^1_0(B_r,\R^k)$ and $\|U-U_t\|_{L^1}\leq t|B_r|$. 
We denote by $C$ any constant which does not depend on $x$ or $t$. 
By the quasi-minimality property of the function $U$ in Theorem \ref{t:lipquasimin} we have
\begin{multline}\label{e:finitperi1}
\int_{B_r}\big(A\nabla u_i\cdot\nabla u_i - A\nabla u_{i,t}\cdot\nabla u_{i,t}\big) + \Lambda\big(|\{|U|>0\}\cap B_r|-|\{|U_t|>0\}\cap B_r|\big) \\
\leq C\|U-U_t\|_{L^1}\int_DA\nabla U_t\cdot\nabla U_t \leq Ct.
\end{multline}
Since $\eta=1$ in $B_{r/2}$ we have $\nabla u_{i,t}=\nabla u_i\ind_{\{|u_i|\geq t\}}$ in $B_{r/2}$ and hence
\begin{equation*}
\int_{B_{r/2}}\big(A\nabla u_i\cdot\nabla u_i - A\nabla u_{i,t}\cdot\nabla  u_{i,t}\big) = \int_{\{0<|u_i|<t\}\cap B_{r/2}}A\nabla u_i\cdot\nabla u_i.
\end{equation*}
On the other hand, with an easy computation we get
\begin{align*}
&\int_{B_r\backslash B_{r/2}}\big(A\nabla u_i\cdot\nabla u_i - A\nabla u_{i,t}\cdot\nabla  u_{i,t}\big) = \int_{\{u_i\geq t\}\cap(B_r\backslash B_{r/2})}\big(2tA\nabla u_i\cdot\nabla\eta - t^2A\nabla\eta\cdot\nabla\eta\big) \\
&\qquad\qquad\qquad+ \int_{\{|u_i|< t\}\cap(B_r\backslash B_{r/2})}\big(\eta(2-\eta)A\nabla u_i\cdot\nabla u_i - u_i^2A\nabla\eta\cdot\nabla\eta +2(1-\eta)u_iA\nabla u_i\cdot\nabla\eta\big) \\
&\qquad\qquad\qquad+ \int_{\{u_i\geq -t\}\cap(B_r\backslash B_{r/2})}\big(-2tA\nabla u_i\cdot\nabla\eta - t^2A\nabla\eta\cdot\nabla\eta\big) \geq -Ct.
\end{align*}
Moreover, since $\eta\neq 1$ in $B_r\backslash B_{r/2}$ and by definition of $u_{i,t}$ we have
\begin{multline*}
|\{|U|>0\}\cap B_r|-|\{|U_t|>0\}\cap B_r| = |\{|U|>0\}\cap B_{r/2}|-|\{|U_t|>0\}\cap B_{r/2}| \\
= |\{0\leq |u_i|\leq t\}\cap\{|U|>0\}\cap B_{r/2}| 
\geq |\{0<|u_i|<t\}\cap B_{r/2}|.
\end{multline*}
Then, we now get from \eqref{e:finitperi1} that
\begin{equation}\label{e:finitperi2}
\int_{\{0<|u_i|<t\}\cap B_{r/2}}A\nabla u_i\cdot\nabla u_i + \Lambda|\{0<|u_i|<t\}\cap B_{r/2}| \leq Ct
\end{equation}
and therefore we have
\begin{align*}
\int_{\{0<|u_i|<t\}\cap B_{r/2}}&|\nabla u_i| \leq \int_{\{0<|u_i|<t\}\cap B_{r/2}}\big(|\nabla u_i|^2+1\big) \\
&\leq \max\{\lambda_{\text{\tiny\sc A}}^2,\Lambda^{-1}\}\bigg(\int_{\{0<|u_i|<t\}\cap B_{r/2}}A\nabla u_i\cdot\nabla u_i + \Lambda|\{0<|u_i|<t\}\cap B_{r/2}|\bigg) \leq Ct.
\end{align*}
We now use the co-area formula to rewrite the above inequality as
\begin{equation*}
    \frac{1}{t}\int_{0}^t\text{Per}\big(\{|u_i|>s\};B_{r/2}\big)\,ds \leq C.
\end{equation*}
Therefore, there exists a sequence $t_n\downarrow 0$ such that $\text{Per}\big(\{|u_i|>t_n\};B_{r/2}\big)\leq C$. Passing to the limit we get that $\text{Per}\big(\{|u_i|>0\};B_{r/2}\big)\leq C$, which concludes the proof.
\end{proof}

\subsection{Freezing of the coefficients and non-degeneracy of the eigenfunctions}
The properties of the eigenfunctions on optimal sets in the case where $A=Id$ have already been studied in \cite{mazzoleni-terracini-velichkov-17}. Thus, we perform a change of variables in order to reduce to this case. We prove in the spirit of \cite[Lemma 3.2]{spolaor-trey-velichkov-19} (see also \cite[Proposition 2.4]{trey-19}) that the vector of the first $k$ eigenfunctions is a local quasi-minimizer at the origin of the Alt-Caffarelli functional. We then prove a non-degeneracy property for the vector of the first $k$ eigenfunctions at the boundary of the optimal set. 

We start with some notations which will be used throughout this paper. For $U\in H^1(\R^d,\R^k)$ and $r>0$ we set
\begin{equation*}
J(U,r)=\int_{B_r}|\nabla U|^2 + \Lambda|\{|U|>0\}\cap B_r|.
\end{equation*}
For $x\in D$ we define the function $F_x : \R^d\rightarrow \R^d$ by
\begin{equation}\label{e:def_F}
F_x(\xi):=x+A_x^{\sfrac12}[\xi],\qquad \xi\in\R^d,
\end{equation}
where $A^{\sfrac12}_x \in \text{Sym}_d^+$ denotes the square root matrix of $A_x$ (notice that, by assumption, the matrix $A_x$ is positive definite). 
Moreover, for $U=(u_1,\dots,u_k)\in H^1(\R^d,\R^k)$ we set $U_x=U\circ F_x$ and $u_{x,i}=u_i\circ F_x$, $i=1,\dots,k$.

\begin{prop}\label{p:changevarU}
Let $U\in H^1_0(D,\R^k)$ be the vector of the first $k$ normalized eigenfunctions on $\O^\ast$.
There exist constants $r_0\in(0,1)$ and $C>0$ such that, if $x\in D$ and $r\leq r_0$ satisfy $B_{\lambda_{\text{\tiny\sc A}} r}(x)\subset D$, then
\begin{equation}\label{e:changevarUa}
J(U_x,r)\leq (1+Cr^{\delta_{\text{\tiny\sc A}}})J(\tilde{U},r)+C\|U_x-\tilde{U}\|_{L^1}
\end{equation}
for every $\tilde{U}\in H^1(\R^d,\R^k)$ such that $U_x-\tilde{U}\in H^1_0(B_r,\R^k)$ and $\|\tilde{U}\|_{L^\infty} \leq \|U_x\|_{L^\infty}$.
\end{prop}

\begin{proof}
Let $V\in H^1_0(D,\R^k)$ be such that $\tilde{U}=V\circ F_x$ and set $\rho=\lambda_{\text{\tiny\sc A}} r$. Observe that $U-V\in H^1_0(F_x(B_r))$ and use $V$ as a test function in \eqref{e:quasiminUa} to get
\begin{multline}\label{e:changevarU1}
\int_{F_x(B_r)}A\nabla U\cdot\nabla U + \Lambda|\{|U|>0\}\cap F_x(B_r)|\leq (1+Cr^d)\int_{F_x(B_r)}A\nabla V\cdot\nabla V \\
+ \Lambda|\{|V|>0\}\cap F_x(B_r)|+C\|U-V\|_{L^1}.
\end{multline}
Moreover, since $A$ has H\"{o}lder continuous coefficients and is uniformly elliptic, we have
\begin{equation}\label{e:changevarU2}
J(U_x,r)\leq\det(A_x^{-\sfrac12})\bigg[(1+dc_{\text{\tiny\sc A}}\lambda_{\text{\tiny\sc A}}^2\rho^{\delta_{\text{\tiny\sc A}}})\int_{F_x(B_r)}A\nabla U\cdot\nabla U + \Lambda|\{|U|>0\}\cap F_x(B_r)|\bigg].
\end{equation}
Similarly, we have the estimate from below
\begin{equation}\label{e:changevarU3}
J(\tilde{U},r)\geq\det(A_x^{-\sfrac12})\bigg[(1-dc_{\text{\tiny\sc A}}\lambda_{\text{\tiny\sc A}}^2\rho^{\delta_{\text{\tiny\sc A}}})\int_{F_x(B_r)}A\nabla V\cdot\nabla V + \Lambda|\{|V|>0\}\cap F_x(B_r)|\bigg].
\end{equation}
Combining \eqref{e:changevarU2}, \eqref{e:changevarU1} and \eqref{e:changevarU3} we get
\begin{equation*}
J(U_x,r)\leq (1+dc_{\text{\tiny\sc A}}\lambda_{\text{\tiny\sc A}}^2\rho^{\delta_{\text{\tiny\sc A}}})\bigg[\frac{1+Cr^d}{1-dc_{\text{\tiny\sc A}}\lambda_{\text{\tiny\sc A}}^2\rho^{\delta_{\text{\tiny\sc A}}}}J(\tilde{U},r) +C\|U_x-\tilde{U}\|_{L^1}\bigg] 
\end{equation*}
which gives \eqref{e:changevarUa}.
\end{proof}

We now prove a non-degeneracy property of the function $U_x=U\circ F_x$ using the approach of David an Toro in \cite{david-toro-15} which is a variant of the result in \cite{alt-caffarelli-81}.

\begin{prop}[Non-degeneracy of $U_x$]\label{p:nondeg}
Let $U=(u_1,\dots,u_k)$ be the vector of the $k$ first eigenfunctions on $\O^\ast$. Let $K\subset\O^\ast$ be a compact set. There exist constants $\eta=\eta_K>0$ and $r_K>0$ such that for every $x\in K$ and $r\leq r_K$ we have
\begin{equation*}
\sum_{i=1}^k\aver{\partial B_r}|u_{x,i}| \leq \eta r\qquad\Longrightarrow\qquad U=0\quad\text{in}\quad B_{r/4\lambda_{\text{\tiny\sc A}}}(x).
\end{equation*}
\end{prop}

We will need the following Lemma which, loosely speaking, provides an estimate of the non-subharmonicity of $U_x$.

\begin{lm}\label{l:nondegl}
Let $U=(u_1,\dots,u_k)$ be the vector of the $k$ first eigenfunctions on $\O^\ast$. Let $K\subset\O^\ast$ be a compact set. There exists constants $C_K>0$ and $r_K>0$ such that for every $x\in K$ and $r\leq r_K$ we have
\begin{equation}\label{e:nondegla}
\sum_{i=1}^k\aver{B_r}\big[(u_{x,i}-h_{r,i})^+\big]^2 \leq C_Kr^{2+\delta_{\text{\tiny\sc A}}},
\end{equation}
where $h_{x,i}$ denotes the harmonic extension of the trace of $u_{x,i}$ to $\partial B_r$.
\end{lm}

\begin{proof}
We define the vector $\tilde{U}=(\tilde{u}_1,\dots,\tilde{u}_k)\in H^1(\R^d,\R^k)$ by
\begin{equation*}
\tilde{u}_i=\left\{
\begin{array}{ll}
\min(u_{x,i},h_{x,i})\quad&\text{in}\quad B_r \\
u_{x,i} \quad&\text{in}\quad D\backslash B_r.
\end{array}\right.
\end{equation*}  
Then, using $\tilde{U}$ as a test function in Proposition \ref{p:changevarU} we get (since $U_x$ is locally Lipschitz continuous in $D$ and because we have the inclusion $\{|\tilde{U}|>0\}\subset\{|U_x|>0\}$ since $\tilde{u}_i\leq u_{x,i}$)
\begin{align}\label{e:nondegl1}
\nonumber\int_{B_r}|\nabla U_x|^2 &\leq (1+Cr^{\delta_{\text{\tiny\sc A}}})\int_{B_r}|\nabla\tilde{U}|^2 + Cr^{\delta_{\text{\tiny\sc A}}}|\{|\tilde{U}|>0\}\cap B_r| + C\|U_x-\tilde{U}\|_{L^1(B_r)} \\
&\leq (1+Cr^{\delta_{\text{\tiny\sc A}}})\int_{B_r}|\nabla\tilde{U}|^2 + Cr^{d+\delta_{\text{\tiny\sc A}}}.
\end{align}
We now set $V_i=\{h_{r,i}<u_{x,i}\}$ for every $i=1,\dots,k$, so that by \eqref{e:nondegl1} we have
\begin{equation}\label{e:nondegl2}
\sum_{i=1}^k\int_{V_i}(|\nabla u_{x,i}|^2-|\nabla h_{r,i}|^2) \leq Cr^{\delta_{\text{\tiny\sc A}}}\int_{B_r}|\nabla\tilde{U}|^2 + Cr^{d+\delta_{\text{\tiny\sc A}}}.
\end{equation}
Moreover, we have the following equalities
\begin{equation}\label{e:nondegl3}
\int_{B_r}(|\nabla\tilde{u}_i|^2-|\nabla u_{x,i}|^2) = \int_{V_i}(|\nabla h_{x,i}|^2-|\nabla u_{x,i}|^2) = -\int_{V_i}|\nabla(u_{x,i}-h_{r,i})|^2.
\end{equation}
Indeed, the first equality follows from the definition of $V_i$. For the second one, we set $v_i=\max(u_{x,i},h_{r,i})$ in $B_r$ and $v_i=u_{x,i}$ elsewhere, so that by harmonicity of $h_{r,i}$ we have
\begin{equation*}
0=\int_{B_r}\nabla h_{r,i}\cdot\nabla(v_i-h_{r,i}) = \int_{V_i}\nabla u_{x,i}\cdot\nabla(v_i-h_{r,i}) = \int_{V_i}\nabla h_{r,i}\cdot\nabla u_{x,i} - \int_{V_i}|\nabla h_{r,i}|^2,
\end{equation*}
which gives \eqref{e:nondegl3}. Finally, combining Poincar\'{e} inequality, \eqref{e:nondegl3}, \eqref{e:nondegl2} and using that $U_x$ is Lipschitz continuous we get
\begin{align*}
\sum_{i=1}^k\aver{B_r}\big[(u_{x,i}-h_{r,i})^+\big]^2 &\leq Cr^2\sum_{i=1}^k\aver{B_r}|\nabla(u_{x,i}-h_{r,i})^+|^2 = Cr^{2-d}\sum_{i=1}^k\int_{V_i}|\nabla(u_{x,i}-h_{r,i})|^2 \\
&=Cr^{2-d}\sum_{i=1}^k\int_{V_i}(|\nabla u_{x,i}|^2-|\nabla h_{r,i}|^2) \leq Cr^{2-d}\bigg(Cr^{\delta_{\text{\tiny\sc A}}}\int_{B_r}|\nabla\tilde{U}|^2 + Cr^{d+\delta_{\text{\tiny\sc A}}}\bigg) \\
&= Cr^{2+\delta_{\text{\tiny\sc A}}}\bigg(\aver{B_r}|\nabla\tilde{U}|^2+1\bigg) \leq  Cr^{2+\delta_{\text{\tiny\sc A}}}\bigg(\aver{B_r}|\nabla U_x|^2+1\bigg) \leq Cr^{2+\delta_{\text{\tiny\sc A}}}.
\end{align*}
\end{proof}

\begin{proof}[Proof of Proposition \ref{p:nondeg}]
Let $\eta>0$ be small and assume that 
\begin{equation}\label{e:nondeg4}
\sum_{i=1}^k\aver{\partial B_r}|u_{x,i}| \leq \eta r.
\end{equation}
We first claim that for every $i\in\{1,\dots,k\}$ we have $|u_{x,i}|<4^{d+1}\eta r$ in $B_{r/2}$. Suppose by contradiction that there exists $\xi_0 \in B_{r/2}$ such that $|u_{x,i}(\xi_0)|\geq 4^{d+1}\eta r$. Since $u_{x,i}$ is $L$-Lipschitz continuous (with $L$ depending on $K$), we have for every $\xi \in B_{\eta r/L}(\xi_0)$
\begin{equation}\label{e:nondeg1}
|u_{x,i}(\xi)|\geq |u_{x,i}(\xi_0)| - |u_{x,i}(\xi)-u_{x,i}(\xi_0)| \geq  (4^{d+1}-1)\eta r.
\end{equation}
Moreover, if $\eta\leq L/4$, by Poisson formula we have for every $\xi \in B_{\eta r/L}(\xi_0) \subset B_{3r/4}$
\begin{equation}\label{e:nondeg2}
|h_{r,i}(\xi)|=\frac{r^2-|\xi|^2}{d\omega_d r}\bigg|\int_{\partial B_r} \frac{u_{x,i}(\tilde{\xi})}{|\xi-\tilde{\xi}|^d}\,d\mathcal{H}^{d-1}(\tilde{\xi})\bigg| \leq \frac{r}{d\omega_d}\Big(\frac{4}{r}\Big)^d\int_{\partial B_r}|u_{x,i}| \leq 4^d\eta r.
\end{equation}
Therefore, using \eqref{e:nondeg1} and \eqref{e:nondeg2} it follows that
\begin{align*}
\aver{B_r}\Big[(u_{x,i}-h_{r,i})^+\Big]^2 &\geq \Big(\frac{\eta}{L}\Big)^d\aver{B_{\eta r/L}(\xi_0)}(u_{x,i}-h_{r,i})^2 \geq \Big(\frac{\eta}{L}\Big)^d\aver{B_{\eta r/L}(\xi_0)}(|u_{x,i}|-|h_{r,i}|)^2 \\
&\geq \Big(\frac{\eta}{L}\Big)^d\Big[(4^{d+1}-1-4^d)\eta r\Big]^2 \geq \frac{\eta^{d+2}}{L^d}r^2,
\end{align*}
which is in contradiction with \eqref{e:nondegla} if $r$ is small enough.

Now, let $\varphi\in C^\infty(B_r)$ be a smooth function such that $0\leq\varphi\leq 1$, $\varphi=1$ in $B_{r/2}$, $\varphi=0$ in $B_r\backslash B_{3r/4}$ and $|\nabla\varphi|\leq Cr$. We set for $i=1,\dots,k$
\begin{equation*}
\tilde{u}_i=\left\{
\begin{array}{ll}
(u_{x,i}-4^{d+1}\eta r\varphi)^+ -(u_{x,i}+4^{d+1}\eta r\varphi)^- \quad&\text{in}\quad B_r \\
u_{x,i} \quad&\text{in}\quad D\backslash B_r,
\end{array} \right.
\end{equation*}
and $\tilde{U}=(\tilde{u}_1,\dots,\tilde{u}_k)\in H^1(\R^d,\R^k)$. Notice that we have $U_x-\tilde{U}\in H^1(B_r,\R^k)$. Moreover, by the preceding claim we have the inclusion $\{|\tilde{U}|>0\}\cap B_r \subset \{|U_x|>0\}\cap (B_r\backslash B_{r/2})$. Therefore, Proposition \ref{p:changevarU} applied to the vector $\tilde{U}$ gives
\begin{equation}\label{e:nondeg3}
\Lambda|\{|U_x|>0\}\cap B_{r/2}| \leq \int_{B_r}(|\nabla\tilde{U}|^2-|\nabla U_x|^2) + Cr^{\delta_{\text{\tiny\sc A}}}J(\tilde{U},r) + C\|U_x-\tilde{U}\|_{L^1(B_r)}.
\end{equation}
By the definition of $\tilde{U}$ and since $U_x$ is $L$-Lipschitz continuous, we have in the ball $B_r$
\begin{equation*}
|\nabla\tilde{U}|^2 \leq |\nabla U_x|^2 + 2.4^{d+1}k\eta rL|\nabla\varphi| + 4^{2(d+1)}\eta^2r^2|\nabla\varphi|^2 \leq |\nabla U_x|^2 + C\eta\quad\text{in}\quad B_r.
\end{equation*}
Since once again $U_x$ is Lipschitz continuous, \eqref{e:nondeg3} now gives
\begin{equation}\label{e:nondeg5}
\Lambda|\{|U_x|>0\}\cap B_{r/2}| \leq C\eta r^d+ Cr^{\delta_{\text{\tiny\sc A}}}\big((L^2+C\eta)r^d +\Lambda\omega_dr^d\big) + CLr^{d+1} \leq C(\eta+r^{\delta_{\text{\tiny\sc A}}})r^d.
\end{equation}
Then, using once again the claim, we deduce that
\begin{equation}\label{e:nondeg6}
\sum_{i=1}^k\int_{B_{r/2}}|u_{x,i}|=\sum_{i=1}^k\int_{B_{r/2}\cap\{|U_x|>0\}}|u_{x,i}| \leq 4^{d+1}k\eta r|\{|U_x|>0\}\cap B_{r/2}| \leq C(\eta+r^{\delta_{\text{\tiny\sc A}}})\eta r^{d+1}.
\end{equation}
Let $y\in B_{r/4\lambda_{\text{\tiny\sc A}}}(x)$. We will find by induction a sequence of radii $r_j$ such that the estimate \eqref{e:nondeg4} holds with the radius $r_j$ and at the point $y$. Let us choose $r_1\in(\frac{r}{8\lambda_{\text{\tiny\sc A}}^2},\frac{r}{4\lambda_{\text{\tiny\sc A}}^2})$ such that
\begin{equation*}
\int_{\partial B_{r_1}}\sum_{i=1}^k|u_{y,i}| \leq \frac{8\lambda_{\text{\tiny\sc A}}^2}{r} \int_{r/8\lambda_{\text{\tiny\sc A}}^2}^{r/4\lambda_{\text{\tiny\sc A}}^2}ds\int_{\partial B_s}\sum_{i=1}^k|u_{y,i}|
\end{equation*}
Then, by \eqref{e:nondeg6} (and since $F_y^{-1}\circ F_x(B_{r/4\lambda_{\text{\tiny\sc A}}^2}) \subset B_{r/2})$ we get
\begin{align*}
\sum_{i=1}^k\aver{\partial B_{r_1}}|u_{y,i}| &\leq Cr^{1-d}\sum_{i=1}^k\int_{\partial B_{r_1}}|u_{y,i}| \leq  Cr^{-d}\sum_{i=1}^k\int_{r/8\lambda_{\text{\tiny\sc A}}^2}^{r/4\lambda_{\text{\tiny\sc A}}^2}ds\int_{\partial B_s}|u_{y,i}| \\
&\leq Cr^{-d} \sum_{i=1}^k \int_{B_{r/4\lambda_{\text{\tiny\sc A}}^2}} |u_{y,i}| = Cr^{-d}\sum_{i=1}^k \int_{F_x^{-1}\circ F_y(B_{r/4\lambda_{\text{\tiny\sc A}}^2})}|u_{x,i}|\,|\det(F_y^{-1}\circ F_x)| \\
&\leq Cr^{-d} \sum_{i=1}^k \int_{B_{r/2}}|u_{x,i}| \leq C(\eta+r^{\delta_{\text{\tiny\sc A}}})\eta r_1 \leq \eta r_1,
\end{align*}
where the last inequality holds if $\eta$ and $r$ are small enough. Therefore, by the same above argument we use to get \eqref{e:nondeg5} and \eqref{e:nondeg6} we now deduce that
\begin{equation*}
|\{|U_y|>0\}\cap B_{r_1/2}| \leq C(\eta+r_1^{\delta_{\text{\tiny\sc A}}})r_1^d
\end{equation*}
and
\begin{equation*}
\sum_{i=1}^k\int_{B_{r_1/2}}|u_{y,i}| \leq C(\eta+r_1^{\delta_{\text{\tiny\sc A}}})\eta r_1^{d+1}.
\end{equation*}
We now choose $r_2\in (\frac{r_1}{4},\frac{r_1}{2})$ such that
\begin{equation*}
\sum_{i=1}^k\aver{\partial B_{r_2}}|u_{y,i}| \leq Cr_1^{-d}\sum_{i=1}^k\int_{r_1/4}^{r_1/2}ds\int_{\partial B_s}|u_{y,i}| \leq Cr_1^{-d}\sum_{i=1}^k\int_{B_{r_1/2}}|u_{y,i}| \leq C(\eta+r_1^{\delta_{\text{\tiny\sc A}}})\eta r_1 \leq \eta r_1,
\end{equation*}
provided that $\eta$ and $r$ are small enough. By induction it follows that there exists a sequence of radii $(r_j)_j$ such that $r_j \in (\frac{r_j}{4},\frac{r_j}{2})$ and
\begin{equation}\label{e:nondeg7}
|\{|U_y|>0\}\cap B_{r_j/2}| \leq C(\eta+r_j^{\delta_{\text{\tiny\sc A}}})\eta r_j.
\end{equation}
Now, if $|U_y|(0)>0$, then $|U_y|>0$ in a neighborhood of $0$ since $U_y$ is continuous, which is in contradiction with \eqref{e:nondeg7} for $\eta$ small enough and $j$ big enough. Hence $|U|(y)=|U_y|(0)=0$ for every $y\in B_{r/4\lambda_{\text{\tiny\sc A}}}(x)$, that is, $U=0$ in $B_{r/4\lambda_{\text{\tiny\sc A}}}(x)$. 
\end{proof}

\begin{oss}[$L^\infty$ non-degeneracy of $U$]\label{r:LinftynondegU}
A consequence of Proposition \ref{p:nondeg} is that $U$ also enjoys the following non-degeneracy property: there exist $\eta=\eta_K>0$ and $r_K>0$ such that for every $x\in K$ and $r\leq r_K$ we have
\begin{equation*}
\|U\|_{L^\infty(B_{\lambda_{\text{\tiny\sc A}} r}(x))}\leq\eta r \qquad\Longrightarrow\qquad U=0\quad\text{in}\quad B_{r/4\lambda_{\text{\tiny\sc A}}}(x).
\end{equation*}
\end{oss}

\subsection{Non-degeneracy of the first eigenfunction and density estimate}
We prove that the first eigenfunction $u_1$ on an optimal set $\O^\ast$ to \eqref{e:shapeopt} is non degenerate at every point of the boundary of $\O^\ast_1$, where $\O^\ast_1$ denotes any connected component of $\O^\ast$ where $u_1$ is positive. The proof follows an idea of Kriventsov and Lin taken from \cite{kriventsov-lin-18}. As a consequence, we obtain that $u_1$ behaves like the distance function to the boundary and also a density estimate for the optimals sets. Obviously, these properties only hold in $\O^\ast_1$, that is, where $u_1$ is positive. However, as pointed out in Remark \ref{r:redu_Oast1}, it is enough to restrict ourselves to this case in order to get the regularity of the whole optimal set $\O^\ast$. 

\begin{prop}[Non-degeneracy of $u_1$]\label{p:nondegu1}
There exists a constant $C_1>0$ such that $C_1 u_1\geq |U|$ in $\O^\ast_1$.
\end{prop}

We first recall the following standard result which is a consequence of \cite[Lemma 2.1]{trey-19}.

\begin{lm}\label{l:degiorgi}
Let $\O\subset D$ be a (non-empty) quasi-open set, $f\in L^\infty(D)$, $f\geq 0$, and $u\in H^1(D)$ be such that $u\geq 0$ on $\partial\O$ and
\begin{equation*}
\dive(A \nabla u)\leq f\quad\text{in}\quad\Omega. 
\end{equation*}
Then, there exists a constant $C>0$, depending only on $d$ and $\lambda_{\text{\tiny\sc A}}$, such that
$$\|u^-\|_{L^\infty(\O)}\le C|\{u<0\}\cap\O|^{\sfrac2d}\|f\|_{L^\infty(\O)}.$$
\end{lm}

\begin{proof}
Set $\O^-=\{u<0\}\cap\O$ and notice that $u\in H^1_0(\O^-)$ Let $v\in H^1_0(\O^-)$ be the solution of $\dive(A \nabla v)= f$ in $\O^-$. By the weak maximum principle we have $v\leq 0$ in $\O^-$ (since $f\geq 0$) and $v\leq u$ in $\O^-$; in particular, $u^-\leq v^-=-v$ in $\O^-$. The proof now follows from Lemma 2.1 in \cite{trey-19} (applied to $-v$).
\end{proof}

\begin{proof}[Proof of Proposition \ref{p:nondegu1}]
We first claim that $\dive(A\nabla|U|)\geq -C|U|$ in $\O^\ast$. 
Let $\varphi\in H^1_0(\O^\ast)$, $\varphi\geq 0$. We use an approximation by mollifiers $A^\varepsilon=(a_{ij}^\varepsilon)$ where $a_{ij}^\varepsilon=a_{ij}\ast\rho_\varepsilon$, and we compute
\begin{align*}
\langle\dive(A^\varepsilon\nabla|U|),\varphi\rangle &=-\sum_{i,j}\int a_{ij}^\varepsilon\partial_i|U|\partial_j\varphi = -\sum_{i,j,l}\int a_{ij}^\varepsilon\partial_iu_l\frac{u_l}{|U|}\partial_j\varphi \\
&=\sum_{i,j,l}\int \partial_j(a_{ij}^\varepsilon\partial_i u_l)\frac{u_l}{|U|}\varphi + \sum_{i,j,l}\int a_{ij}^\varepsilon\partial_iu_l\partial_j\bigg(\frac{u_l}{|U|}\bigg)\varphi \\
&=-\sum_{i,j,l}\int a_{ij}^\varepsilon\partial_iu_l\partial_j\bigg(\frac{u_l}{|U|}\varphi\bigg) + \sum_{i,j,l,p}\int a_{ij}^\varepsilon\partial_iu_l\bigg(\frac{\partial_ju_l}{|U|} - \frac{u_lu_p}{|U|^3}\partial_ju_p\bigg)\varphi.
\end{align*}
Therefore, passing to the limit as $\varepsilon\rightarrow 0$ we get
\begin{align}\label{e:nondegux11}
\nonumber\langle\dive(A\nabla|U|),\varphi\rangle &\geq \sum_{l}\int\dive(A\nabla u_l)\frac{u_l}{|U|}\varphi +\sum_{l,p}\int\frac{1}{|U|^3}\big(u_l^2A\nabla u_l\cdot\nabla u_l - u_lu_pA\nabla u_l\cdot \nabla u_p\big)\varphi \\
&\geq -\sum_{l}\lambda_l(\O^\ast)\int b\frac{u_l^2}{|U|}\varphi \geq -\lambda_k(\O^\ast)c_b\int|U|\varphi,
\end{align}
which proves the claim.

Let $r_0>0$ be small (to be chosen soon) and set $\O_r=\{x\in\O^\ast_1\,:\,|U(x)|<r\}$ for every $r>0$. Since $u_1>0$ in $\O^\ast_1$ we have $m:=\inf\{u_1(x)\,:\,x\in\O^\ast_1,\, |U(x)|=r_0\}>0$. We set $M_0=m^{-1}r_0$ and $v_0=M_0u_1-|U|$. The claim implies that $\dive(A\nabla v_0)\leq C|U|$ in $\O_{r_0}$. Moreover, by construction of $v_0$ we have $v_0\geq 0$ on $\partial\O_{r_0}$. Therefore, by Lemma \ref{l:degiorgi} we get 
\begin{equation*}
-\inf_{\O_{r_0}}v_0 = \|v_0^-\|_{L^\infty(\O_{r_0})} \leq C|\{v_0<0\}\cap\O_{r_0}|^{\sfrac2d}\|U\|_{L^\infty(\O_{r_0})}.
\end{equation*}
Then, from \eqref{e:finitperi2} (and a compactness argument) we have $|\O_{r_0}|\leq Cr_0$ so that we deduce from the above inequality that $-\inf_{\O_{r_0}}v_0 \leq \overline{C}r_0^{1+\sfrac2d}$ for some $\overline{C}>0$ independent of $r_0$. Therefore, in $\O_{r_0}\backslash\O_{r_0/2}$ we have
\begin{equation*}
M_0u_1=|U|+v_0\geq|U|-\overline{C}r_0^{1+\sfrac2d} \geq\big(1-2\overline{C}r_0^{\sfrac2d}\big)|U|\quad\text{in}\quad \O_{r_0}\backslash\O_{r_0/2}.
\end{equation*}
We now choose $r_0$ small enough so that $4\overline{C}r_0^{\sfrac2d}\leq1$ and set $M_1=\big[1-2\overline{C}r_0^{\sfrac2d}\big]^{-1}M_0$ and $v_1=M_1u_1-|U|$. It follows that $v_1\geq0$ in $\O_{r_0}\backslash\O_{r_0/2}$; in particular we have $v_1\geq0$ on $\partial\O_{r_0/2}$ and hence the above argument now applies to $v_1$ in $\O_{r_0/2}$. Therefore, an induction gives that $v_k\geq 0$ in $\O_{r_{k-1}}\backslash\O_{r_k}$ for every $k\geq 1$, where we have set $v_k=M_ku_1-|U|$, $M_k=\big[1-2\overline{C}r_{k-1}^{\sfrac2d}\big]^{-1}M_{k-1}$ and $r_k=2^{-k}r_0$. Moreover, we have
\begin{equation*}
\log(M_k)=\log(M_0)-\sum_{i=1}^k\log\big[1-2\overline{C}r_{i-1}^{\sfrac2d}\big]\leq\log(M_0)+C\sum_{i=1}^k2^{-2i/d}\leq C+\log(M_0)
\end{equation*}
and hence $M_k\leq CM_0$. It follows that $|U|\leq M_ku_1\leq CM_0u_1$ in $\O_{r_{k-1}}\backslash\O_{r_k}$ for every $k\geq 0$ and therefore that $|U|\leq CM_0u_1$ in $\O_{r_0}$. On the other hand, since $\inf_{\O^\ast_1\backslash\O_{r_0}}u_1>0$, there exists $M>0$ such that $|U|\leq Mu_1$ in $\O^\ast_1\backslash\O_{r_0}$. This completes the proof.
\end{proof}

We now prove that the first eigenfunction on an optimal set has the same growth than the distance function near the boundary. This property will be useful to prove that the boundaries of blow-up sets Hausdorff convergence to the boundary of the blow-up limit set.

\begin{prop}[Uniform growth of $u_1$ at the boundary]\label{p:growth_boundary}
Let $K\subset D$ be a compact set. There exist constants $c_K>0$ and $r_K>0$ such that the following growth condition holds
\begin{equation*}
u_1(x) \geq c_K\,dist(x,\partial\O^\ast_1)\qquad\text{for every } x\in \O^\ast_1\cap K\text{ such that } dist(x,\partial\O^\ast_1)\leq r_K.
\end{equation*}
\end{prop}

\begin{proof}
We set $r=(2\lambda_{\text{\tiny\sc A}})^{-1}dist(x,\partial\O^\ast_1)$ and we denote by $h_{x,1}$ the harmonic extension of the trace of $u_{x,1}$ to $\partial B_r$. 
By non degeneracy of $u_1$ (Propositions \ref{p:nondegu1} and \ref{p:nondeg}) we have (and because $h_{x,1}$ is harmonic)
\begin{equation}\label{e:growth_boundary3}
h_{x,1}(0) = \aver{\partial B_r} h_{x,1} = \aver{\partial B_r} u_{x,1} \geq \frac{1}{C_1}\aver{\partial B_r} |U_x| \geq \frac{\eta}{\sqrt{k}C_1} r =:\eta_1r.
\end{equation}
Therefore, with the triangle inequality we get
\begin{equation}\label{e:growth_boundary1}
u_1(x) = u_{x,1}(0) \geq h_{x,1}(0) - |u_{x,1}(0)-h_{x,1}(0)| \geq \eta_1 r -|u_{x,1}(0)-h_{x,1}(0)|.
\end{equation}
We now want to estimate $|u_{x,1}(0)-h_{x,1}(0)|$ is terms of $r$.
We apply Proposition \ref{p:changevarU} to the test function $\tilde{U}=(h_{x,1},u_{x,2},\dots,u_{x,k})$ and get (since $u_{x,1}$ is Lipschitz continuous and that $|u_{x,1}|>0$ in $B_r$)
\begin{equation}\label{e:growth_boundary2}
\int_{B_r}|\nabla(u_{x,1}-h_{x,1})|^2 = \int_{B_r}\big(|\nabla u_{x,1}|^2 -|\nabla h_{x,1}|^2\big) \leq Cr^{d+\delta_{\text{\tiny\sc A}}}.
\end{equation}
Now, let $\tau>0$ be small to be chosen soon. Since $u_{x,1}$ and $h_{x,1}$ are Lipschitz continuous, we have for every $\xi\in B_{\tau r}$
\begin{align*}
|u_{x,1}(0)-h_{x,1}(0)| &\leq |u_{x,1}(0)-u_{x,1}(\xi)| + |u_{x,1}(\xi)-h_{x,1}(\xi)| + |h_{x,1}(\xi)-h_{x,1}(0)| \\&\leq C\tau r + |u_{x,1}(\xi)-h_{x,1}(\xi)|.
\end{align*}
Moreover, using Poincar\'{e} inequality to the function $u_{x,1}-h_{x,1}$ and the estimate \eqref{e:growth_boundary2}, we have
\begin{align*}
|u_{x,1}(0)-h_{x,1}(0)| &\leq C\tau r + \aver{B_{\tau r}}|u_{x,1}(\xi)-h_{x,1}(\xi)| \leq C\tau r + \tau^{-d}\aver{B_{r}}|u_{x,1}(\xi)-h_{x,1}(\xi)| \\
&\leq C\tau r + C\tau^{-d}r\aver{B_{r}}|\nabla(u_{x,1}(\xi)-h_{x,1}(\xi))| \\
&\leq  C\tau r + C\tau^{-d}r\bigg(\aver{B_{r}}|\nabla(u_{x,1}(\xi)-h_{x,1}(\xi))|^2\bigg)^{1/2} \\
&\leq  \big(C\tau + C\tau^{-d}r^{\delta_{\text{\tiny\sc A}}/2}\big)r \leq \frac{\eta_1}{2}r,
\end{align*}
where the last inequality holds by choosing first $\tau$ small enough and then $r_K$ (depending on $\tau$) small enough. In view of \eqref{e:growth_boundary1}, Proposition \ref{p:nondegu1} now follows.
\end{proof}

\begin{prop}[Density estimate for $\O^\ast_1$]\label{p:density_est}
Let $U$ be the vector of the first $k$ normalized eigenfunctions on $\O^\ast$ and let $K\subset D$ be a compact set. There exist constants $r_K>0$ and $c_K\in(0,1)$ such that for every $x_0\in \partial\O^\ast_1\cap K$ and $r\leq r_K$ we have 
\begin{equation*}
c_K|B_r|\leq|\O^\ast_1\cap B_r(x_0)|\leq (1-c_K)|B_r|.
\end{equation*}
\end{prop}

\begin{proof}
We first prove that we have
\begin{equation}\label{e:density_est1}
c|B_r|\leq|\{|U_{x_0}|>0\}\cap B_r|\leq (1-c)|B_r|.
\end{equation}
The first inequality follows from the non-degeneracy of $U_{x_0}$ (Proposition \eqref{p:nondeg}) since it implies that there exists $\xi\in\partial B_{r/2}$ such that $\sum_{i=1}^k|u_{x_0,i}(\xi)|\geq \frac{\eta r}{2}$, and hence, using that $U_{x_0}$ is $L$-Lipschitz continuous, that
\begin{equation*}
|U_{x_0}|\geq\frac{1}{\sqrt{k}}\sum_{i=1}^k|u_{x_0,i}|\geq \frac{\eta r}{4\sqrt{k}}\quad\text{in}\quad B_{\frac{\eta r}{4L}}(\xi).
\end{equation*}
For the second estimate, consider the test function $\tilde{U}=(h_{r,1},u_{x_0,2},\dots,u_{x_0,k})\in H^1(\R^d,\R^k)$, where $h_{r,1}$ denotes as usual the harmonic extension of $u_{x_0,1}$ to $\partial B_r$, and note that by the strong maximum principle we have $h_{r,1}>0$ in $B_r$ since $u_{x_0,1}$ is non-negative. Then, by Proposition \ref{p:changevarU} applied to $\tilde{U}$, and since $u_{x_0,1}$ is $L$-Lipschitz continuous, we get
\begin{align}\label{e:densityest1}
\nonumber\int_{B_r}\big(|\nabla u_{x_0,1}|^2-|\nabla h_{r,1}|^2\big) &\leq \Lambda|\{|U_{x_0}|=0\}\cap B_r| + Cr^{\delta_{\text{\tiny\sc A}}}J(\tilde{U},r) + C\|u_{x_0,1}-h_{r,1}\|_{L^1} \\
&\leq \Lambda|\{|U_{x_0}|=0\}\cap B_r| +Cr^{d+\delta_{\text{\tiny\sc A}}}.
\end{align}
Moreover, by Proposition \ref{p:nondegu1} and the harmonicity of $h_{r,1}$ (and also because $u_{x_0,1}(0)=0$), we have
\begin{equation}\label{e:densityest2}
|u_{x_0,1}(0)-h_{r,1}(0)|=h_{r,1}(0)=\aver{\partial B_r}h_{r,1}=\aver{\partial B_r}u_{x_0,1} \geq \eta_1 r,
\end{equation}
where $\eta_1$ is defined as in \eqref{e:growth_boundary3}.
Now, let $\tau>0$ be small. Since $h_{r,1}$ is $2L$-Lipschitz continuous we have for every $\xi\in B_{\tau r}$
\begin{align*}
|u_{x_0,1}(0)-h_{r,1}(0)| &\leq |u_{x_0,1}(0)-u_{x_0,1}(\xi)| + |u_{x_0,1}(\xi)-h_{r,1}(\xi)|+|h_{r,1}(\xi)-h_{r,1}(0)| \\
&\leq 3L\tau r + |u_{x_0,1}(\xi)-h_{r,1}(\xi)|.
\end{align*} 
Then, averaging over $B_{\tau r}$ and using \eqref{e:densityest2} leads to
\begin{equation}\label{e:densityest3}
\eta_1 r\leq|u_{x_0,1}(0)-h_{r,1}(0)| \leq 3L\tau r +\aver{B_{\tau r}}|u_{x_0,1}-h_{r,1}|.
\end{equation}
Moreover, by Poincar\'{e} inequality and Cauchy-Schwarz inequality we have
\begin{align*}\label{e:densityest4}
\aver{B_{\tau r}}|u_{x_0,1}-h_{r,1}|& \leq \tau^{-d}\aver{B_r}|u_{x_0,1}-h_{r,1}|\leq \tau^{-d}r\aver{B_r}|\nabla(u_{x_0,1}-h_{r,1})| \\
&\leq \tau^{-d}r^{1-\frac{d}{2}}\bigg(\int_{B_r}|\nabla(u_{x_0,1}-h_{r,1})|^2\bigg)^{1/2} = \tau^{-d}r^{1-\frac{d}{2}}\bigg(\int_{B_r}|\nabla u_{x_0,1}|^2-|\nabla h_{r,1}|^2\bigg)^{1/2}
\end{align*}
which combined with \eqref{e:densityest1} and \eqref{e:densityest3}, and after some rearrangements, gives
\begin{equation*}
2\Lambda r^{-d}|\{|U_{x_0}|=0\}\cap B_r| \geq \eta_1^2\tau^{2d} -C\tau^{2d+2} -Cr^{\delta_{\text{\tiny\sc A}}}.
\end{equation*}
Then choose $\tau$, depending only on $\eta_1$ and $C$, small enough so that $C\tau^{2d+2}\leq \eta_1^2\tau^{2d}/2$ and then choose $r$, depending only on $\eta_1, \tau$ and $C$, such that $Cr^{\delta_{\text{\tiny\sc A}}}\leq\eta_1^2\tau^{2d}/4$ to conclude the proof.

Now, by a change of variables, the density estimate in \eqref{e:density_est1} gives
\begin{equation*}
c|A^{\sfrac12}_{x_0}[B_r]| \leq |\{|U|>0\}\cap A^{\sfrac12}_{x_0}[B_r]|\leq (1-c)|A^{\sfrac12}_{x_0}[B_r]|.
\end{equation*}
Then set $c_K=\lambda_{\text{\tiny\sc A}}^{-2d}c$ so that (because we have the inclusions $B_{\lambda_{\scalebox{.5}{A}}^{-1}r}\subset A^{\sfrac12}_{x_0}[B_r]\subset B_{\lambda_{\text{\tiny\sc A}} r}$)
\begin{equation*}
    c_K|B_{\lambda_{\text{\tiny\sc A}} r}|=c|B_{\lambda_{\scalebox{.5}{A}}^{-1}r}|\leq c|A^{\sfrac12}_{x_0}[B_r]| \leq |\{|U|>0\}\cap A^{\sfrac12}_{x_0}[B_r]| \leq |\{|U|>0\}\cap B_{\lambda_{\text{\tiny\sc A}} r}|.
\end{equation*}
Similarly we have
\begin{align*}
    |\{|U|=0\}\cap B_{\lambda_{\text{\tiny\sc A}} r}| \geq  |\{|U|=0\}\cap A^{\sfrac12}_{x_0}[B_r]| &= |A^{\sfrac12}_{x_0}[B_r]| - |\{|U|>0\}\cap A^{\sfrac12}_{x_0}[B_r]|  \\
    &\geq c|A^{\sfrac12}_{x_0}[B_r]| \geq c|B_{\lambda_{\scalebox{.5}{A}}^{-1}r}|=c_K|B_{\lambda_{\text{\tiny\sc A}} r}|,
\end{align*}
which concludes the proof.
\end{proof}

\subsection{Weiss monotonicity formula}\label{sub:weiss}
We prove a monotonicity formula for the vector of the first $k$ eigenfunctions on an optimal set $\O^\ast$. The proof follows the idea of \cite[Theorem 1.2]{weiss-99} (see also \cite[Proposition 3.1]{mazzoleni-terracini-velichkov-17}).
For every $U\in H^1(\R^d,\R^k)$ and $r>0$ we define
\begin{equation*}
W(U,r)= \frac{1}{r^d}J(U,r)-\frac{1}{r^{d+1}}\int_{\partial B_r}|U|^2.
\end{equation*}

\begin{prop}\label{p:weissder}
Let $U=(u_1,\dots,u_k)$ be the vector of the first $k$ normalized eigenfunctions on $\O^\ast$ and let $K\subset D$ be a compact set.  
Then there exist constants $r_K>0$ and $C_K>0$ such that for every $x_0\in\partial\O^\ast\cap K$ and every $r\leq r_K$ the function $U_{x_0}=U\circ F_{x_0}=(u_{x_0,1},\dots,u_{x_0,k})$ satisfies
\begin{equation}\label{e:weissdera}
\frac{d}{dr}W(U_{x_0},r)\geq \frac{1}{r^{d+2}}\sum_{i=1}^k\int_{\partial B_r}|x\cdot\nabla u_{x_0,i}-u_{x_0,i}|^2\,dx-C_Kr^{\delta_{\text{\tiny\sc A}}-1}.
\end{equation}
Moreover, the limit $\lim_{r\to 0^+}W(U_{x_0},r)$ exists and is finite.
\end{prop}

\begin{proof}
We first compare $U_{x_0}$ with its one-homogeneous extension in the ball $B_r$, namely the one-homogeneous function $\tilde{U}=(\tilde{u}_1,\dots,\tilde{u}_k):B_r\to\R^k$ defined by $\tilde{U}(\xi)=\frac{|\xi|}{r}U_{x_0}\Big(\frac{r}{|\xi|}\xi\Big)$. We have
\begin{equation*}
\int_{B_r}|\nabla\tilde{U}|^2 = \int_{B_r}\bigg[|\nabla_\theta U_{x_0}|^2+\frac{|U_{x_0}|^2}{r^2} \bigg] \bigg(\frac{r}{|\xi|}\xi\bigg)d\xi = \frac{r}{d}\int_{\partial B_r}\bigg[|\nabla_\theta U_{x_0}|^2+\frac{|U_{x_0}|^2}{r^2} \bigg]
\end{equation*}
and for the measure term
\begin{equation*}
|\{|\tilde{U}|>0\}\cap B_r|=\frac{r}{d}\mathcal{H}^{d-1}(\{|U|>0\}\cap\partial B_r).
\end{equation*}
Then, we use $\tilde{U}$ as a test function in \eqref{e:changevarUa} which gives
\begin{align}\label{e:weissder1}
\nonumber J(U_{x_0},r) &\leq J(\tilde{U},r) +C\Big(r^{\delta_{\text{\tiny\sc A}}}J(\tilde{U},r) +\|U_{x_0}-\tilde{U}\|_{L^1}\Big) \\
&\leq \frac{r}{d}\int_{\partial B_r}\bigg[|\nabla_\theta U_{x_0}|^2+\frac{|U_{x_0}|^2}{r^2} \bigg] + \Lambda\frac{r}{d}\mathcal{H}^{d-1}(\{|U|>0\}\cap\partial B_r) + C_0r^{d+\delta_{\text{\tiny\sc A}}}
\end{align}
for some $C_0\geq C(2\omega_d\|\nabla U_{x_0}\|_{L^\infty}^2+\Lambda\omega_d+2\omega_d\|\nabla U_{x_0}\|_{L^\infty})$ where the constant $C$ is given by Proposition \ref{p:changevarU}.
We now compute the derivative of $W(U_{x_0},r)$ and use \eqref{e:weissder1} to obtain
\begin{align*}
\frac{d}{dr}W(U_{x_0},r) &= \frac{1}{r^d}\bigg(\int_{\partial B_r}|\nabla U_{x_0}|^2+\Lambda\mathcal{H}^{d-1}(\{|U_{x_0}|>0\}\cap\partial B_r) \bigg) -\frac{d}{r^{d+1}}J(U_{x_0},r) \\
& \quad\qquad\qquad\qquad\qquad\qquad\qquad+\frac{2}{r^{d+2}}\int_{\partial B_r}|U_{x_0}|^2 -\frac{1}{r^{d+1}}\sum_{i=1}^k\int_{\partial B_r}2u_{{x_0},i}\frac{\partial u_{x_0,i}}{\partial\nu} \\
&\geq \frac{1}{r^d}\int_{\partial B_r}\left|\frac{\partial U_{x_0}}{\partial\nu}\right|^2 + \frac{1}{r^{d+2}}\int_{\partial B_r}|U_{x_0}|^2 -\frac{1}{r^{d+1}}\sum_{i=1}^k\int_{\partial B_r}2u_{x_0,i}\frac{\partial u_{x_0,i}}{\partial\nu} - dC_0r^{\delta_{\text{\tiny\sc A}}-1} \\
&= \frac{1}{r^{d+2}}\sum_{i=1}^k\int_{\partial B_r}\bigg[r^2\left|\frac{\partial u_{x_0,i}}{\partial\nu}\right|^2 + u_{x_0,i}^2 - 2ru_{x_0,i}\frac{\partial u_{x_0,i}}{\partial\nu}\bigg] - dC_0r^{\delta_{\text{\tiny\sc A}}-1} \\
&= \frac{1}{r^{d+2}}\sum_{i=1}^k\int_{\partial B_r}|x\cdot\nabla u_{x_0,i}-u_{x_0,i}|^2-dC_0r^{\delta_{\text{\tiny\sc A}}-1},
\end{align*}
which is \eqref{e:weissdera}. This also proves that the function $r\mapsto W(U_{x_0},r)+\frac{d}{\delta_{\text{\tiny\sc A}}}C_0r^{\delta_{\text{\tiny\sc A}}}$ is non-decreasing and hence that the limit of $W(U_{x_0},r)$ as $r$ tend to $0$ exists. Moreover, this limit is finite since we have the bound
\begin{equation*}
W(U_{x_0},r)\geq -\frac{1}{r^{d+1}}\int_{\partial B_r}|U_{x_0}|^2 \geq -d\omega_d\|\nabla U_{x_0}\|_{L^\infty}^2\qquad\text{for every } r>0.
\end{equation*}
\end{proof}

As a consequence of the previous result, we get a monotonicity formula for global minimizers of the Alt-Caffarelli functional.

\begin{definition}\label{d:minaltcaff}
We say that $U\in H^1(\R^d,\R^k)$ is a global minimizer of the (vectorial) Alt-Caffarelli functional
\begin{equation*}
J(U)=\int_{\R^d}|\nabla U|^2 + \Lambda|\{|U|>0\}|
\end{equation*}
if $J(U,r)\leq J(\tilde{U},r)$ for every $r>0$ and every $\tilde{U}\in H^1(\R^d,\R^k)\cap L^\infty(\R^d,\R^k)$ such that $U-\tilde{U}\in H^1_0(B_r,\R^k)$.
\end{definition}

\begin{prop}\label{p:minAFonehom}
Let $U=(u_1,\dots,u_k)\in H^1(\R^d,\R^k)$ be a global minimizer of the Alt-Caffarelli functional $J$ such that $U(0)=0$. Then we have
\begin{equation*}
\frac{d}{dr}W(U,r)\geq \frac{1}{r^{d+2}}\sum_{i=1}^k\int_{\partial B_r}|x\cdot\nabla u_{i}-u_{i}|^2.
\end{equation*}
In particular, if $r\mapsto W(U,r)$ is constant in $(0,+\infty)$, then $U$ is a one-homogeneous function.
\end{prop}

\begin{proof}
Since $U$ is a global minimizer of $J$, it satisfies \eqref{e:changevarUa} with $C=0$ and hence the computations in the proof of Proposition \ref{p:weissder} hold with $C_0=0$. The last claim of the proposition follows from the fact that $x\cdot\nabla u_i=u_i$ in $\R^d$ implies that $u_i$ is one-homogeneous.
\end{proof}

\section{Blow-ups}\label{s:blow-up}
In this section we study the blow-ups limits (at the origin) of the functions $U_{x_0}=U\circ F_{x_0}$, where $x_0\in\partial\O^\ast\cap D$. Throughout this section, $U$ will denote the first $k$ normalized eigenfunctions on the optimal set $\O^\ast=\{|U|>0\}$. We prove that the blow-up limits are one-homogeneous and global minimizers of the Alt-Caffarelli functional. As a consequence, we also prove that the boundaries of two connected components of $\O^\ast$ have an empty intersection in $D$. 

Let $(x_n)_{n\in\N}$ be a sequence of points on $\partial\O^\ast\cap D$ converging to some $x_0\in \partial\O^\ast\cap D$ and let $(r_n)_{n\in\N}$ be a sequence of positive radii tending to $0$. 
Since $U$ is Lipschitz continuous, up to extracting a subsequence, the sequence defined by
\begin{equation*}\label{e:defblowupseq}
B_{x_n,r_n}(\xi)=\frac{1}{r_n}U(x_n+r_n\xi),\qquad \xi\in\R^d,
\end{equation*}
converges locally uniformly to a Lipschitz continuous function $B_0\in  H^1_{\text{loc}}(\R^d,\R^k)$. We will often set $B_n=B_{x_n,r_n}$ and deal with this  sequence in a new set of coordinates, that is, we will consider the sequence $\tilde{B}_n$ defined by
\begin{equation*}
\tilde{B}_n(\xi)=B_n\circ A^{\sfrac12}_{x_n}(\xi)=\frac{1}{r_n}U_{x_n}(r_n\xi),\qquad \xi\in\R^d.    
\end{equation*}

\begin{definition}
If $B_{x_n,r_n}$ converges locally uniformly in $\R^d$ to some $B_0$, we say that $B_{x_n,r_n}$ is a blow-up sequence (with fixed center if $x_n=x_0$ for every $n\geq 1$). If the center is fixed, we say that $B_0$ is a blow-up limit at $x_0$. We denote by $\mathcal{BU}_U(x_0)$ the space of all blow-up limits at $x_0$.
\end{definition}

We start with a standard result on the convergence of the blow-up sequences and we give the details of the proofs for convenience of the reader. Recall that $\O^\ast_1$ stands for any connected component of $\O^\ast$ where the first eigenfunction $u_1$ is positive.

\begin{prop}[Convergence of the blow-up sequences]\label{p:convblowup}
Let $(x_n)_{n\in\N}\subset \partial\O^\ast\cap D$ be a sequence converging to some $x_0\in\partial\O^\ast\cap D$, $r_n\rightarrow 0$ and assume that the blow-up sequence $B_n:=B_{x_n,r_n}$ converges locally uniformly to $B_0\in H^1_{\text{loc}}(\R^d,\R^k)$.
Then, up to a subsequence, we have
\begin{enumerate}
\item The sequence $B_n$ converges to $B_0$ strongly in $H^1_{\text{loc}}(\R^d,\R^k)$.
\item The sequences of characteristic functions $\ind_{\O_n}$ converges in $L^1_{\text{loc}}(\R^d)$ to the characteristic function $\ind_{\O_0}$, where we have set $\O_n=\{|B_n|>0\}$ and $\O_0=\{|B_0|>0\}$.
\item The function $B_0$ is non-degenerate: there exits a constant $\eta_0>0$ such that for every every $y\in\overline{\O}_0$ we have
\begin{equation*}
\|B_0\|_{L^\infty(B_{r}(y))}\geq \eta_0 r\quad\text{for every}\quad r>0.
\end{equation*}
\item If $x_0\in\partial\O^\ast_1\cap D$, then the sequences of closed sets $\overline{\O}_n$ and $\O_n^c$ converge locally Hausdorff to $\overline{\O}_0$ and $\O_0^c$ respectively.
\end{enumerate}
\end{prop}

\begin{proof}
Notice that it is enough to prove that the sequence $\tilde{B}_n= B_n\circ A^{\sfrac12}_{x_n}$ strongly converges to $\tilde{B}_0:=B_0\circ A^{\sfrac12}_{x_0}$ in $H^1_{\text{loc}}(\R^d,\R^k)$ and that $\ind_{\{|\tilde{B}_n|>0\}}$ converges to $\ind_{\{|\tilde{B}_0|>0\}}$ in $L^1_{\text{loc}}(\R^d)$ to prove the parts {\it (1)} and {\it (2)} of Proposition \ref{p:convblowup}.

Since $\tilde{B}_n$ is uniformly Lipschitz, $\tilde{B}_n$ converges, up to a subsequence, weakly in $H^1_{\text{loc}}(\R^d,\R^k)$ and strongly in $L^2_{\text{loc}}(\R^d,\R^k)$ to $\tilde{B}_0$. Moreover, the local uniform convergence of $|\tilde{B}_n|$ to $|\tilde{B}_0|$ implies that $\ind_{\{|\tilde{B}_0|>0\}}\leq\liminf_{n\rightarrow\infty}\ind_{\{|\tilde{B}_n|>0\}}$. Therefore, it is sufficient to prove that for every ball $B_r\subset \R^d$ we have
\begin{equation}\label{e:convblowup5}
\limsup_{n\rightarrow+\infty}\bigg(\int_{B_r}|\nabla\tilde{B}_n|^2 +\Lambda|\{|\tilde{B}_n|>0\}\cap B_r|\bigg) \leq \int_{B_r}|\nabla\tilde{B}_0|^2 +\Lambda|\{|\tilde{B}_0|>0\}\cap B_r|.
\end{equation}
Let $\varphi\in C^{\infty}_c(\R^d)$ be a smooth function such that $0\leq\varphi\leq 1$, $\{\varphi=1\}=B_r$ and $\varphi=0$ outside $B_{2r}$. 
We set $\tilde{U}_n=\varphi \tilde{B}_0+(1-\varphi)\tilde{B}_n\in H^1(\R^d,\R^k)$ and notice that we have
\begin{equation*}
U_{x_0}-\tilde{U}_n^{r_n}\in H^1_0(B_{2rr_n},\R^k) \quad\text{where}\quad \tilde{U}_n^{r_n}(\xi)=r_n\tilde{U}_n\Big(\frac{1}{r_n}\xi\Big), \,\xi\in\R^d.
\end{equation*}
Then, using $\tilde{U}_n^{r_n}$ as a test function in Proposition \ref{p:changevarU} and by a change of variables we get
\begin{align}\label{e:convblowup1}
\nonumber\int_{B_{2r}}|\nabla\tilde{B}_n|^2 + \Lambda|\{|\tilde{B}_n|>0\}\cap B_{2r}|&\leq (1+C(rr_n)^{\delta_{\text{\tiny\sc A}}})\bigg(\int_{B_{2r}}|\nabla\tilde{U}_n|^2 + \Lambda|\{|\tilde{U}_n|>0\}\cap B_{2r}|\bigg) \\
&\qquad\qquad\qquad\qquad\qquad+ r_nC\|\varphi(\tilde{B}_0-\tilde{B}_n)\|_{L^1}.
\end{align}
Since we have $\tilde{U}_n=\tilde{B}_n$ in $\{\varphi=0\}$ and $\tilde{U}_n=\tilde{B}_0$ in $\{\varphi=1\}$, it follows that
\begin{equation*}
|\{|\tilde{U}_n|>0\}\cap B_{2r}| \leq |\{|\tilde{B}_n|>0\}\cap\{\varphi =0\}\cap B_{2r}| + |\{|\tilde{B}_0|>0\}\cap\{\varphi=1\}| + |\{0<\varphi<1\}|,
\end{equation*}
so that \eqref{e:convblowup1} now gives
\begin{multline}\label{e:convblowup2}
\int_{\{\varphi>0\}}\big(|\nabla\tilde{B}_n|^2-|\nabla\tilde{U}_n|^2\big) + \Lambda\big(|\{|\tilde{B}_n>0\}\cap\{\varphi>0\}|-|\{\tilde{B}_0|>0\}\cap\{\varphi=1\}|\big) 
\leq \Lambda|\{0<\varphi<1\}| \\+  C(rr_n)^{\delta_{\text{\tiny\sc A}}}\bigg(\int_{B_{2r}}|\nabla\tilde{U}_n|^2 + \Lambda|\{|\tilde{U}_n|>0\}\cap B_{2r}|\bigg) +  r_nC\|\varphi(\tilde{B}_0-\tilde{B}_n)\|_{L^1}.
\end{multline}
Now, since $\tilde{B}_n$ converges strongly in $L^2(B_{2r})$ we have that
\begin{align}\label{e:convblowup3}
\nonumber\limsup_{n\rightarrow+\infty}&\int_{\{\varphi>0\}}\big(|\nabla\tilde{B}_n|^2-|\nabla\tilde{U}_n|^2\big) = \limsup_{n\rightarrow+\infty}\int_{\{\varphi>0\}}\big(|\nabla\tilde{B}_n|^2-|\nabla(\varphi\tilde{B}_0 +(1-\varphi)\tilde{B}_n)|^2\big)\\
\nonumber&=\limsup_{n\rightarrow+\infty}\int_{\{\varphi>0\}}\big(|\nabla\tilde{B}_n|^2-|(\tilde{B}_0-\tilde{B}_n)\nabla\varphi+(1-\varphi)\nabla\tilde{B}_n+\varphi\nabla\tilde{B}_0|^2\big) \\
&=\limsup_{n\rightarrow+\infty}\int_{\{\varphi>0\}}\big((1-(1-\varphi)^2)|\nabla\tilde{B}_n|^2-2\varphi(1-\varphi)\nabla\tilde{B}_n\cdot\nabla\tilde{B}_0-\varphi^2|\nabla\tilde{B}_0|^2 \big) \\
\nonumber&=\limsup_{n\rightarrow+\infty}\int_{\{\varphi>0\}}(1-(1-\varphi)^2)\big(|\nabla\tilde{B}_n|^2-|\nabla\tilde{B}_0|^2\big),
\end{align}
and since $\nabla\tilde{B}_n$ converges weakly in $L^2(\{0<\varphi<1\})$ to $\tilde{B}_0$ we have that
\begin{equation}\label{e:convblowup6}
\int_{\{0<\varphi<1\}}(1-(1-\varphi)^2)|\nabla\tilde{B}_0|^2 \leq \limsup_{n\rightarrow+\infty}\int_{\{0<\varphi<1\}}(1-(1-\varphi)^2)|\nabla\tilde{B}_n|^2. 
\end{equation}
Therefore, \eqref{e:convblowup6} and \eqref{e:convblowup3} now entail that
\begin{align}\label{e:convblowup4}
\nonumber\limsup_{n\rightarrow+\infty}\int_{\{\varphi=1\}}\big(|\nabla\tilde{B}_n|^2-|\nabla\tilde{B}_0|^2\big) 
&\leq\limsup_{n\rightarrow+\infty}\int_{\{\varphi>0\}}(1-(1-\varphi)^2)\big(|\nabla\tilde{B}_n|^2-|\nabla\tilde{B}_0|^2\big) \\
&\leq \limsup_{n\rightarrow+\infty}\int_{\{\varphi>0\}}\big(|\nabla\tilde{B}_n|^2-|\nabla\tilde{U}_n|^2\big).
\end{align}
Finally, combining \eqref{e:convblowup4} and \eqref{e:convblowup2} we get
\begin{align*}
&\limsup_{n\rightarrow+\infty}\bigg(\int_{\{\varphi=1\}}\big(|\nabla\tilde{B}_n|^2-|\nabla\tilde{B}_0|^2\big) + \Lambda\big(|\{|\tilde{B}_n|>0\}\cap\{\varphi=1\}|-|\{|\tilde{B}_0|>0\}\cap\{\varphi=1\}|\big)\bigg) \\
&\leq \limsup_{n\rightarrow+\infty}\bigg(\int_{\{\varphi>0\}}\big(|\nabla\tilde{B}_n|^2-|\nabla\tilde{U}_n|^2\big) + \Lambda\big(|\{|\tilde{B}_n|>0\}\cap\{\varphi>0\}|-|\{|\tilde{B}_0|>0\}\cap\{\varphi=1\}|\big)\bigg) \\
&\leq \Lambda|\{0<\varphi<1\}|.
\end{align*}
Since we can choose $\varphi$ so that $|\{0<\varphi<1\}|$ is arbitrary small, this proves \eqref{e:convblowup5} and concludes the proof of parts {\it (1)} and {\it (2)} of Proposition \ref{p:convblowup}.

We now prove part {\it (3)}. Let $y \in \overline{\O}_0$ and $r>0$. There exists $z\in B_r(y)$ such that $|B_0|(z)>0$, and hence such that $|B_n|(z)>0$ for $n$ large enough. Therefore, $U\neq 0$ in $B_{rr_n}(x_n+r_nz)$ and hence, by the non-degeneracy of $U$ (Remark \ref{r:LinftynondegU}), we get that
\begin{equation*}
r_n\|B_n\|_{L^\infty(B_{4\lambda_{\text{\tiny\sc A}}^2 r}(z))} = \|U\|_{L^\infty(B_{4\lambda_{\text{\tiny\sc A}}^2rr_n}(x_n+r_nz))} \geq 4\lambda_{\text{\tiny\sc A}}\eta rr_n.
\end{equation*}
In particular, there exists $z_n\in B_{4\lambda_{\text{\tiny\sc A}}^2 r}(z)$ such that $|B_n|(z_n)\geq4\lambda_{\text{\tiny\sc A}}\eta r$. Up to a subsequence, $z_n$ converges to some $z_\infty\in \overline{B_{4\lambda_{\text{\tiny\sc A}}^2 r}(z)}$ and, since $B_n$ uniformly converges to $B_0$, we have that
\begin{equation*}
\|B_0\|_{L^\infty(B_{(4\lambda_{\text{\tiny\sc A}}^2+1)r}(y))}\geq\|B_0\|_{L^\infty(B_{4\lambda_{\text{\tiny\sc A}}^2 r}(z))} \geq |B_0|(z_\infty) =\lim_{n\rightarrow+\infty}|B_n|(z_n) \geq 4\lambda_{\text{\tiny\sc A}}\eta r,
\end{equation*}
which gives {\it (3)}. The proof of the Hausdorff convergence of the free boundaries is standard and follows from the non-degeneracy of $U$ and $B_0$, and the growth property of $U$ near the boundary of $\O^\ast_1$ (see Proposition \ref{p:growth_boundary}).
\end{proof}

\begin{lm}[Optimality of the blow-up limits]\label{l:optblowup}
Let $(x_n)_{n\in\N}\subset \partial\O^\ast\cap D$ be a sequence converging to some $x_0\in\partial\O^\ast\cap D$, $r_n\rightarrow 0$ and assume that the blow-up sequence $B_n:=B_{x_n,r_n}$ converges to some $B_0\in H^1_{\text{loc}}(\R^d,\R^k)$ in the sense of Proposition \ref{p:convblowup}.
Then $\tilde{B}_0:=B_0\circ A_{x_0}^{\sfrac12}$ is a global minimizer of the Alt-Caffarelli functional $J$ (see definition \ref{d:minaltcaff}).
\end{lm}

\begin{proof}
Let $r>0$ and $\tilde{U}\in H^1_{\text{loc}}(\R^d,\R^k)\cap L^\infty(\R^d,\R^k)$ be such that $\tilde{B}_0-\tilde{U}\in H^1_0(B_r,\R^k)$. Let $\eta\in C^\infty_c(B_r)$ be such that $0\leq\eta\leq 1$ and set $\tilde{B}_n=B_n\circ A_{x_n}^{\sfrac12}$ and  $\overline{V}_n=\tilde{U}+(1-\eta)(\tilde{B}_n-\tilde{B}_0)$. Consider the test function $V_n\in H^1(\R^d,\R^k)$ defined by $V_n(\xi)=r_n\overline{V}_n(r_n^{-1}\xi)$ and note that $U_{x_0}-V_n\in H^1_0(B_{rr_n},\R^k)$ (since we have $\tilde{B}_n-\overline{V}_n\in H^1_0(B_r,\R^k)$). By Proposition \ref{p:changevarU} applied to $V_n$ and a change of variables it follows that
\begin{multline}\label{e:optblowup1}
\int_{B_r}|\nabla\tilde{B}_n|^2 + \Lambda|\{|\tilde{B}_n|>0\}\cap B_r| \leq (1+C(rr_n)^{\delta_{\text{\tiny\sc A}}})\bigg(\int_{B_r}|\nabla \overline{V}_n|^2 
+ \Lambda|\{|\overline{V}_n|>0\}\cap B_r|\bigg) \\
+ Cr_n \|\tilde{B}_n-\overline{V}_n\|_{L^1(B_r)}.
\end{multline}
Note that from {\it (1)} and {\it (2)} of Proposition \ref{p:convblowup} we deduce that $\overline{V}_n$ converges strongly in $H^1_{\text{loc}}$ to $\tilde{U}$ and that $\ind_{\{|\overline{V}_n|>0\}}$ converges strongly in $L^1_{\text{loc}}$ to $\ind_{\{|\tilde{U}|>0\}}$. Moreover, since $\overline{V}_n=\tilde{U}$ in $\{\eta=1\}$, we have the estimate
\begin{equation*}
|\{|\overline{V}_n|>0\}\cap B_r|  \leq 
|\{|\tilde{U}|>0\}\cap B_r| + |\{\eta\neq 1\}\cap B_r|.
\end{equation*}
Therefore, passing to the limit as $n\to\infty$ in \eqref{e:optblowup1} we get 
\begin{equation*}
\int_{B_r}|\nabla\tilde{B}_0|^2 + \Lambda|\{|\tilde{B}_0|>0\}\cap B_r| \leq \int_{B_r}|\nabla\tilde{U}|^2 + \Lambda\big(|\{|\tilde{U}|>0\}\cap B_r| + |\{\eta\neq 1\}\cap B_r|\big).
\end{equation*}
Since we can choose $\eta$ such that that $|\{\eta\neq 1\}\cap B_r|$ is arbitrary small, this gives that $J(\tilde{B}_0,r) \leq J(\tilde{U},r)$ and concludes the proof.
\end{proof}

As a consequence of the Weiss almost-monotonicity formula we get that the blow-up sequences with fixed center converge to a one-homogeneous function. 

\begin{lm}[Homogeneity of the blow-up limits]\label{l:homblowup}
For every $x_0\in\partial\O^\ast\cap D$, the blow-up limits $B_0\in\mathcal{BU}_U(x_0)$  are one-homogeneous functions. 
\end{lm}

\begin{proof}
Let $B_n=B_{x_0,r_n}$ converging (in the sense of Proposition \ref{p:convblowup}) to $B_0$. In particular, $\tilde{B}_n$ converges strongly in $H^1_{\text{loc}}$ and in $L^1_{\text{loc}}$ to $\tilde{B}_0$ which implies that $\lim_{n\rightarrow+\infty}W(\tilde{B}_n,r)=W(\tilde{B}_0,r)$. Moreover, by Proposition \ref{p:weissder} the limit $\lim_{s\rightarrow 0^+}W(U_{x_0},s)$ exists and is finite. Therefore, we have for every $r>0$
\begin{equation}\label{e:homblowup1}
W(\tilde{B}_0,r)=\lim_{n\rightarrow+\infty}W(\tilde{B}_n,r)=\lim_{n\rightarrow+\infty}W(U_{x_0},rr_n)=\lim_{s\rightarrow 0^+}W(U_{x_0},s),
\end{equation}
which says that the function $r\mapsto W(\tilde{B}_0,r)$ is constant on $(0,+\infty)$. Then, it follows from Lemma \ref{l:optblowup} and Proposition \ref{p:minAFonehom} that $\tilde{B}_0$, and hence $B_0$, is one-homogeneous.
\end{proof}

We now reduce to the scalar case. More precisely, we prove that for any blow-up limit $B_0\in\mathcal{BU}_U(x_0)$, the function $|\tilde{B}_0|=|B_0\circ A^{\sfrac{1}{2}}_{x_0}|$ is a global minimizer of the scalar Alt-Caffarelli functional
\begin{equation}\label{e:defscalarAF}
    H^1_{\text{loc}}(\R^d) \ni u \to \int_{\R^d}|\nabla u|^2 + \Lambda|\{u>0\}|.
\end{equation}

\begin{lm}\label{l:blowupxi}
Let $x_0\in\partial\O^\ast\cap D$,  $B_0\in\mathcal{BU}_U(x_0)$ and set $\tilde{B}_0=B_0\circ A^{\sfrac12}_{x_0}$. Then there exists a unit vector $\xi\in\partial B_1\subset \R^k$ such that $\tilde{B}_0=|\tilde{B}_0|\xi$.
\end{lm}

\begin{proof}
Set $S=\partial B_1\cap\{|\tilde{B}_0|>0\}$. By Lemma \ref{l:homblowup}, the components of $\tilde{B}_0=(b_1,\dots,b_k)$ are one-homogeneous functions and by Lemma \ref{l:optblowup}, they are harmonic on the cone $\{|\tilde{B}_0|>0\}$. Therefore, in polar coordinates we have $b_i(r,\theta)=r\varphi_i(r)$ where $\varphi_i$ is solution of the equation
\begin{equation*}
-\Delta_{\mathbb{S}^{d-1}}\varphi_i=(d-1)\,\varphi_i\quad\text{in}\quad S,\qquad\varphi_i=0\quad\text{on}\quad\partial S,
\end{equation*}
where $\Delta_{\mathbb{S}^{d-1}}$ stands for the Laplace-Beltrami operator. By Proposition \ref{p:density_est}, the components of $\tilde{B}_0$ are not all zero. Therefore, at least one $\varphi_i$ is non-zero and hence $d-1$ is an eigenvalue of $-\Delta_{\mathbb{S}^{d-1}}$ on $S$. Since the functions $\varphi_i$ are non-negative, it follows that $\lambda_1(S)=d-1$, where $\lambda_1(S)$ denotes the first eigenvalue on $S$. Moreover, by Lemma \ref{l:optblowup} we have $|S|<|\partial B_1|$ and by \cite[Remark 4.8]{mazzoleni-terracini-velichkov-17} it follows that the first eigenvalue $\lambda_1(S)$ is simple. Then, there exists non-negative constants $\alpha_1,\dots,\alpha_k$, not all zero, such that $\varphi_i=\alpha_i\varphi$ where $\varphi$ stands for the normalized eigenfunction of $-\Delta_{\mathbb{S}^{d-1}}$ on $S$. Now set $\alpha=(\alpha_1,\dots,\alpha_k)$ so that we have $\tilde{B}_0=\varphi \alpha$ on $\partial B_1$. Since $|\alpha|\neq 0$, setting $\xi=|\alpha|^{-1}\alpha$ we have $\tilde{B}_0=|\tilde{B}_0|\xi$ on $\partial B_1$ and hence on $\R^d$ by one-homogeneity.
\end{proof}

\begin{lm}\label{l:minscalAC}
Let $x_0\in\partial\O^\ast\cap D$,  $B_0\in\mathcal{BU}_U(x_0)$ and set $\tilde{B}_0=B_0\circ A^{\sfrac12}_{x_0}$.
Then the function $|\tilde{B}_0|$ is a global minimizer of the (scalar) Alt-Caffarelli functional defined in \eqref{e:defscalarAF}.
\end{lm}

\begin{proof}
Let $r>0$ and $\tilde{u}\in H^1_{\text{loc}}(\R^d)\cap L^\infty(\R^d)$ be such that $|\tilde{B}_0|-\tilde{u}\in H^1_0(B_r)$. Since $\tilde{B}_0=|\tilde{B}_0|\xi$ by Lemma \ref{l:blowupxi}, we have that $\tilde{B}_0-\tilde{u}\,\xi=(|\tilde{B}_0|-\tilde{u})\,\xi\in H^1_0(B_r,\R^k)$ and hence, by optimality of $\tilde{B}_0$ (see Lemma \ref{l:optblowup}) we have
\begin{equation*}
\int_{B_r}|\nabla|\tilde{B}_0||^2 + \Lambda|\{|\tilde{B}_0|>0\}\cap B_r| = J(\tilde{B}_0,r) \leq J(\tilde{u}\,\xi,r) 
= \int_{B_r}|\nabla\tilde{u}|^2 + \Lambda|\{|\tilde{u}|>0\}\cap B_r|.
\end{equation*}
\end{proof}

We conclude this section with a consequence of the one-homogeneity and the optimality of $|\tilde{B}_0|$ which states that two connected components of an optimal set cannot meet inside $D$. It is then enough to prove the regularity of one connected component $\O^\ast_1$ of $\O^\ast$ and hence to reduce to a one-phase free boundary problem (see Proposition \ref{l:opt_cond_u1}).

\begin{prop}\label{p:connected_comp_sep}
Denote by $(\O^\ast_i)_{i=1}^l$ the $l\leq k$ connected componenents of an optimal set $\O^\ast$ for \eqref{e:shapeopt}. Then, we have $\partial\O^\ast_i\cap\partial\O^\ast_j\cap D=\emptyset$ for every $i,j\in\{1,\dots,l\}$, $i\neq j$.
\end{prop}

\begin{proof}
Let $x_0\in\partial\O^\ast_i\cap\partial\O^\ast_j\cap D$. Since $\sigma(\O^\ast_i)\subset\sigma(\O^\ast)$, there exists $k_i\in\{1,\dots,k-1\}$ such that $\lambda_s(\O^\ast_i)\in\{\lambda_1(\O^\ast),\dots,\lambda_k(\O^\ast)\}$ for every $s=1,\dots,k_i$ and $\lambda_s(\O^\ast_i)\notin\{\lambda_1(\O^\ast),\dots,\lambda_k(\O^\ast)\}$ for every $s>k_i$. It follows that $\O^\ast_i$ is solution of the problem \eqref{e:shapeopt} with $k=k_i$ and $D=D\setminus(\overline{\O^\ast\setminus\O^\ast_i})$. Similarly, for some $k_j\in\{1,\dots,k-1\}$, $\O^\ast_j$ is solution of \eqref{e:shapeopt} with $k=k_j$. Then, we denote by $V=(v_1,\dots,v_{k_i})$ and $W=(w_1,\dots,w_{k_j})$ the eigenfunctions on $\O^\ast_i$ and $\O^\ast_j$ respectively. Let $r_n\rightarrow 0$ and define the blow-up sequences
\begin{equation*}
B_n(\xi)=\frac{1}{r_n}U(x_0+r_n\xi),\quad B_n^V(\xi)=\frac{1}{r_n}V(x_0+r_n\xi),\quad B_n^W(\xi)=\frac{1}{r_n}W(x_0+r_n\xi),\qquad \xi\in\R^d.
\end{equation*}
Up to a subsequence, $B_n$, $B_n^V$ and $B_n^W$ converge to some blow-up limits $B_0\in\mathcal{BU}_U(x_0)$, $B_0^V\in\mathcal{BU}_V(x_0)$ and $B_0^W\in\mathcal{BU}_W(x_0)$. By Lemmas \ref{l:homblowup} and \ref{l:minscalAC}, $|\tilde{B}_0^V|$ and $|\tilde{B}_0^W|$ are non-trivial, one-homogeneous and global solutions of the Alt-Caffarelli functional. Therefore, the density at the origin of each set $\{|\tilde{B}_0^V|>0\}$ and $\{|\tilde{B}_0^W|>0\}$ is at least $1/2$ (see \cite[Lemma 5]{mazzoleni-terracini-velichkov-17}) and, since all the components of $\tilde{B}_0^V$ and $\tilde{B}_0^W$ are among the ones of $\tilde{B}_0$, it follows that $|\{|\tilde{B}_0|>0\}\cap B_1|=|B_1|$. Hence, $|\tilde{B}_0|$ is harmonic in $B_1$ since it minimizes the Alt-Caffarelli functional. And since $|\tilde{B}_0|$ is also a non-trivial and non-negative function which vanishes at $0$, this gives a contradiction (by the maximum principle).
\end{proof}

\section{Regularity of the free boundary}\label{s:reg}
This section is devoted to the proof of Theorem \ref{t:main}. Recall that we denote by $\O^\ast$ a solution to the problem \eqref{e:shapeopt} and that $\O^\ast_1$ stands for any connected component of $\O^\ast$ where the first eigenfunction is positive. 

\subsection{The optimality condition on the free boundary}
We prove that the vector $U$ of the first $k$ eigenfunctions on $\O^\ast$ satisfies an optimality condition on the boundary $\partial\O^\ast\cap D$ in the sense of the viscosity.

\begin{definition}
Let $D\subset\R^d$ be an open set and $U:D\subset\R^d\rightarrow\R^k$ be a continuous function.

$\bullet$ We say that $\varphi\in\mathcal{C}(D)$ touches $|U|$ by below (resp. by above) at $x_0\in D$ if $\varphi(x_0)=|U(x_0)|$ and $\varphi\leq|U|$ (resp. $\varphi\geq|U|$) in a neighborhood of $x_0$.
    
$\bullet$ Let $\O\subset D$ be an open set and let $g : D\rightarrow\R$ be continuous and non-negative function. We say that $U$ satisfies the boundary condition
\begin{equation}\label{e:boundarycond}
|A^{\sfrac12}[\nabla|U|]|=g\quad\text{on}\quad \partial\O\cap D
\end{equation}
in the viscosity sense if, for every $x_0\in\partial\O\cap D$ and every $\varphi\in\mathcal{C}^2(D)$ such that $\varphi^+:=\max(\varphi,0)$ touches $|U|$ by below (resp. by above) at $x_0$ we have
\begin{equation*}
|A^{\sfrac12}_{x_0}[\nabla\varphi(x_0)]|\leq g(x_0) \qquad \Big(\text{resp.}\quad |A^{\sfrac12}_{x_0}[\nabla\varphi(x_0)]|\geq g(x_0)\Big).
\end{equation*}
    
$\bullet$ Let, moreover, $\lambda=(\lambda_1,\dots,\lambda_k)\in \R^k$ be a vector of positive coordinates. We say that the function $U=(u_1,\dots,u_k)$ is a viscosity solution of the problem
\begin{equation*}
-\dive(A\nabla U)=\lambda bU\quad\text{in}\quad\O,\qquad U=0\quad\text{on}\quad\partial\O\cap D,\qquad |A^{\sfrac12}[\nabla|U|]|=g\quad\text{on}\quad \partial\O\cap D,
\end{equation*}
if for every $i=1,\dots,k$ the component $u_i$ is a solution of the PDE
\begin{equation*}
-\dive(A\nabla u_i)=-\lambda_i bu_i\quad\text{in}\quad\O,\qquad u_i=0\quad\text{on}\quad\partial\O\cap D,
\end{equation*}
and if the boundary condition \eqref{e:boundarycond} holds in the viscosity sense.
\end{definition}

\begin{oss}\label{r:eqdefvisc}
Another equivalent definition of the boundary condition is to say that \eqref{e:boundarycond} holds if for every $x_0\in\partial\O\cap D$ and every $\psi\in\mathcal{C}^2(\R^d)$ such that $\psi^+$ touches $|U_{x_0}|$ by below (resp. by above) at $0$ we have
$|\nabla\psi(0)|\leq g(x_0)$ (resp. $|\nabla\psi(0)|\geq g(x_0)$).
Indeed, if we set $\psi=\varphi\circ F_{x_0}$ then we have $|\nabla\psi(0)|=|A^{\sfrac12}_{x_0}[\nabla\varphi(0)]|$ (see also \cite[Appendix A]{spolaor-trey-velichkov-19}).
\end{oss}

\begin{lm}[Optimality condition on the free boundary]\label{l:opt_cond_U}
Let $U=(u_1,\dots,u_k)$ be the vector of the first $k$ normalized eigenfunctions on $\O^\ast$ and set $\lambda=(\lambda_1(\O^\ast),\dots,\lambda_k(\O^\ast))$. Then $U$ is a viscosity solution of the problem 
\begin{equation}\label{e:viscpb}
-\dive(A\nabla U)=\lambda bU\quad\text{in}\quad\O^\ast,\quad U=0\quad\text{on}\quad\partial\O^\ast\cap D,\quad |A^{\sfrac12}[\nabla|U|]|=\sqrt{\Lambda}\quad\text{on}\quad \partial\O^\ast\cap D.
\end{equation}
\end{lm}

\begin{proof}
Since $U$ is Lipschitz continuous, we only have to prove that the boundary condition holds in the viscosity sense. Let $x_0\in\partial\O^\ast\cap D$ and let $\psi\in\mathcal{C}^2(\R^d)$ be a function touching $|U_{x_0}|$ by below at $0$ (see Remark \ref{r:eqdefvisc}). We fix an infinitesimal sequence $r_n$ and set for every $\xi\in\R^d$
\begin{equation*}
\tilde{B}_n(\xi)=\frac{1}{r_n}U_{x_0}(r_n\xi)\quad\text{and}\quad \psi_n(\xi)=\frac{1}{r_n}\psi(r_n\xi).
\end{equation*}
Up to a subsequence, the blow-up sequences $(\tilde{B}_n)_{n\in\N}$ and $(\psi_n)_{n\in\N}$ converge locally uniformly in $\R^d$ to some $\tilde{B}_0\in H^1_{\text{loc}}(\R^d,\R^k)$ and to $\psi_0(\xi):=\xi\cdot\nabla\psi(0)$ respectively. We can assume that $\nabla\psi(0)=|\nabla\psi(0)|e_d$ (by a change of variables) and that $|\nabla\psi(0)|\neq 0$, since otherwise $|\nabla\psi(0)|\leq \sqrt{\Lambda(x_0)}$ obviously holds. We have $\psi\leq|U_{x_0}|$ near $0$ an hence $\psi_0\leq|\tilde{B}_0|$ in $\R^d$ which gives that $|\tilde{B}_0|>0$ in the half-space $\{x_d>0\}$. Since $|\tilde{B}_0|$ is a one-homogeneous (Lemma \ref{l:homblowup}) and non-degenerate (Proposition \ref{p:convblowup}) function, it follows that $\{\tilde{B}_0>0\}=\{x_d>0\}$ (see \cite[Lemma 5.30]{russ-trey-velichkov-19}). Moreover, $|\tilde{B}_0|$ is a local minimizer of the Alt-Caffarelli functional (Lemma \ref{l:minscalAC}) and hence satisfies the optimality condition 
\begin{equation*}
|\nabla|\tilde{B}_0||=\sqrt{\Lambda}\quad\text{on}\quad\{x_d=0\}.
\end{equation*}
Therefore we have $|\tilde{B}_0(\xi)|=\sqrt{\Lambda}\,\xi_d^+$ and hence $\psi_0(\xi)=|\nabla\psi(0)|\,\xi_d\leq|\tilde{B}_0(\xi)|=\sqrt{\Lambda}\,\xi_d^+$, which completes the proof when $\psi$ touches by below. The case when $\psi$ touches by above is similar.
\end{proof}

\subsection{Regular and singular parts of the optimal sets}
In this section we prove that the regular part of an optimal set $\O^\ast$ (see Definition \ref{d:reg_sing}) is relatively open in $\partial\O^\ast$.

For any set $\O\subset\R^d$ we define the blow-ups sets $\O_{x,r}$ of $\O$ by
\begin{equation*}
\O_{x,r} = \frac{\O -x}{r},\qquad x\in\R^d,\, r>0.
\end{equation*}
Given Lebesgue measurable sets $(\O_n)_{n\in\N}$ and $\O$ in $\R^d$, we say that $\O_n$ locally converges to $\O$, and we write $\O_n\xrightarrow{\text{loc}}\O$, if the sequence of characteristics functions $\ind_{\O_n}$ converges in $L^1_{\text{loc}}$ to $\ind_{\O}$.

\begin{definition}\label{d:reg_sing}
Let $\O\subset D$ be an open set. We define the regular part of $\O$ in $D$ by
\begin{equation*}
\text{Reg}(\partial\O\cap D) = \Big\{x_0\in\partial\O\cap D \, : \, \exists\nu_{x_0}\in\partial B_1\subset\R^d, \, \O_{x_0,r}\xrightarrow{\text{loc}} \{y\in\R^d \, : \, y\cdot\nu_{x_0}\leq 0\} \text{ as } r\rightarrow 0^+\Big\}.
\end{equation*}
The singular part of $\O$ in $D$ is then define by $\text{Sing}(\partial\O\cap D) = \big(\partial\O\cap D\big)\backslash\text{Reg}(\partial\O\cap D)$.
\end{definition}

\begin{lm}\label{l:densityUx}
Let $U=(u_1,\dots,u_k)$ be the vector of the first $k$ normalized eigenfunctions on $\O^\ast$. Then,
\begin{enumerate}
    \item For every $x_0\in\partial\O^\ast\cap D$ the limit 
    \begin{equation}\label{e:densityUxa}
        \Theta_{U_{x_0}}(0):=\lim_{r\rightarrow 0^+}\frac{|\{|U_{x_0}|>0\}\cap B_r|}{|B_r|}
    \end{equation}
    exists and we have
    \begin{equation}\label{e:densityUxb}
        \Theta_{U_{x_0}}(0) = \frac{1}{\Lambda\omega_d}\lim_{r\rightarrow 0^+}W(U_{x_0},r).
    \end{equation}
    
    \item There exists $\delta>0$ such that, for every $x_0\in\partial\O^\ast\cap D$ we have $\Theta_{U_{x_0}}(0)\in\big\{\frac{1}{2}\big\}\cup\big[\frac{1}{2}+\delta,1\big[$.
\end{enumerate}
\end{lm}

\begin{proof}
Let $(r_n)_{n\in\N}$ be an infinitesimal sequence and set $\tilde{B}_n(\xi)=\frac{1}{r_n}U_{x_0}(r_n\xi)$. Up to a subsequence, $\tilde{B}_n$ converges to some $\tilde{B}_0$ (in the sense of Proposition \ref{p:convblowup}). Since $\tilde{B}_0$ is one homogeneous (Lemma \ref{l:homblowup}) and harmonic in $\{|\tilde{B}_0|>0\}$ (Lemma \ref{l:optblowup}) we have 
\begin{equation*}
    \int_{B_r}|\nabla\tilde{B}_0|^2 =\frac{1}{r}\int_{\partial B_r}|\tilde{B}_0|^2,
\end{equation*}
and hence, for every $r>0$, we get that
\begin{align}\label{e:densityUx2}
    \nonumber W(\tilde{B}_0,r)&=\frac{1}{r^d}\int_{B_r}|\nabla\tilde{B}_0|^2 +\frac{\Lambda}{r^d}|\{|\tilde{B}_0|>0\}\cap B_r|-\frac{1}{r^{d+1}}\int_{\partial B_r}|\tilde{B}_0|^2 \\
    &= \Lambda\omega_d\frac{|\{|\tilde{B}_0|>0\}\cap B_r|}{|B_r|}.
\end{align}
On the other hand, by \eqref{e:homblowup1} we have that $W(\tilde{B}_0,r)=\lim_{s\rightarrow 0^+}W(U_{x_0},s)$ for every $r>0$ and therefore
\begin{equation}\label{e:densityUx1}
     \frac{1}{\Lambda\omega_d}\lim_{s\rightarrow 0^+}W(U_{x_0},s) = \frac{|\{|\tilde{B}_0|>0\}\cap B_r|}{|B_r|} \qquad\text{for every}\quad r>0.
\end{equation}
Then, using that $\tilde{B}_n$ converges to $\tilde{B}_0$ in $L^1_{\text{loc}}(\R^d)$, it follows that
\begin{equation*}
    \frac{1}{\Lambda\omega_d}\lim_{s\rightarrow 0^+}W(U_{x_0},s) = \frac{|\{|\tilde{B}_0|>0\}\cap B_1|}{|B_1|} = \lim_{n\rightarrow\infty}\frac{|\{|\tilde{B}_n|>0\}\cap B_1|}{|B_1|} = \lim_{n\rightarrow\infty}\frac{|\{|U_{x_0}|>0\}\cap B_{r_n}|}{|B_{r_n}|}.
\end{equation*}
This proves part {\it (1)} of the Lemma since the above equalities hold for any sequence $r_n\downarrow 0$.

From \eqref{e:densityUx1} and \eqref{e:densityUxb} it follows that the density of the cone $\{|\tilde{B}_0|>0\}$ at $0$ is given by
\begin{equation*}
    \lim_{r\rightarrow 0^+}\frac{|\{|\tilde{B}_0|>0\}\cap B_r|}{|B_r|} = \Theta_{U_{x_0}}(0)\in[0,1].
\end{equation*}
Moreover, $|\tilde{B}_0|$ is a non-trivial (part {\it (3)} of Proposition \ref{p:convblowup}), one-homogeneous (Lemma \ref{l:homblowup}) and harmonic function in $\{|\tilde{B}_0|>0\}$ (Lemma \ref{l:minscalAC}). Therefore, the density of $\{|\tilde{B}_0|>0\}$ at $0$ cannot be strictly less than $\frac12$ (otherwise, setting $S=\{|\tilde{B}_0|>0\}\cap \partial B_1$, the two first parts of \cite[Remark 4.8]{mazzoleni-terracini-velichkov-17} respectively give $\lambda_1(S)\leq d-1$ and $\lambda_1(S)>d-1$)), cannot belong to $\big(\frac12,\frac12+\delta)$ for some universal constant $\delta>0$ (see \cite[Lemma 5.3]{mazzoleni-terracini-velichkov-17}) and is less than $1-c$ by Proposition \ref{p:density_est}.
\end{proof}

We will also need the following characterization of the regular part.

\begin{lm}\label{l:carac_reg_part}
We have 
\begin{equation*}
\text{Reg}(\partial\O^\ast\cap D)=\bigg\{x_0\in \partial\O^\ast\cap D \, : \, \Theta_{U_{x_0}}(0)=\frac12\bigg\},
\end{equation*}
where $\Theta_{U_{x_0}}(0)$ is define in \eqref{e:densityUxa}.
\end{lm}

\begin{proof}
Let $x_0\in \partial\O^\ast\cap D$, $r_n\downarrow 0$ and $B_n=B_{x_0,r_n}$ be a blow-up sequence converging (in the sense of Proposition \ref{p:convblowup}) to some $B_0$; in particular, $\O^\ast_{x_0,r_n} = \{|B_n|>0\}$ locally converges to $\{|B_0|>0\}$. 
By \eqref{e:densityUxb}, \eqref{e:densityUx1} and a change of variables (because $\tilde{B}_0=B_0\circ A^{\sfrac12}_{x_0}$) we have 
\begin{equation}\label{e:carac_reg_part1}
\Theta_{U_{x_0}}(0)=\frac{|\{|\tilde{B}_0|>0\}\cap B_1|}{|B_1|}=\frac{|\{|B_0|>0\}\cap A^{\sfrac12}_{x_0}[B_1]|}{|A^{\sfrac12}_{x_0}[B_1]|}.
\end{equation}
If $x_0\in\text{Reg}(\partial\O^\ast\cap D)$, then $\{|B_0|>0\}$ is an half-space and 
it follows by \eqref{e:carac_reg_part1} that $\Theta_{U_{x_0}}(0)=1/2$.
Reciprocally, assume that $\Theta_{U_{x_0}}(0)=1/2$. It is enough to prove that $\{|\tilde{B}_0|>0\}$ is an half-space, since then $\{|B_0|>0\}$ is also an half-space.
Set $S=\{|\tilde{B}_0|>0\}\cap\partial B_1$ and notice that $\mathcal{H}^{d-1}(S)=d\omega_d/2$ since $|\tilde{B}_0|$ is one homogeneous. Assume by contradiction that $S=S_0\sqcup S_1$ is the disjoint union of two sets $S_0, S_1\subset\partial B_1$. Since $|\tilde{B}_0|$ is one homogeneous and harmonic on $\{|\tilde{B}_0|>0\}$ it follows that $\varphi=|\tilde{B}_0|_{|\partial B_1}$ is solution of
\begin{equation*}
    -\Delta_{\mathbb{S}^{d-1}}\varphi=(d-1)\varphi\quad\text{in}\quad S_0,\qquad \varphi = 0 \quad\text{on}\quad \partial S_0,
\end{equation*}
which implies that $\lambda_1(S_0)\leq d-1$. On the other hand, since $\mathcal{H}^{d-1}(S_0)<d\omega_d/2$, we also have that $\lambda_1(S_0)>d-1$ (see \cite[Remark 4.8]{mazzoleni-terracini-velichkov-17}), which is a contradiction. Therefore, $S$ is connected and hence $\lambda_1(S)=d-1$. This implies that $S$ is, up to a rotation, the half-sphere $\partial B_1\cap\{x_d>0\}$ and hence that $\{|\tilde{B}_0|>0\}$ is the half-space $\{x_d>0\}$.
\end{proof}

\begin{prop}\label{p:Reg_is_open}
The regular set $\text{Reg}(\partial\O^\ast\cap D)$ is an open subset of $\partial\O^\ast$.
\end{prop}

\begin{proof}
Let $x_0\in\text{Reg}(\partial\O^\ast\cap D)$ and assume by contradiction that there exists a sequence $(x_n)_{n\in\N}$ of points in $\text{Sing}(\partial\O^\ast\cap D)=(\partial\O^\ast\cap D)\backslash\text{Reg}(\partial\O^\ast\cap D)$ converging to $x_0$. By Lemmas \ref{l:densityUx} and \ref{l:carac_reg_part} we have $\Theta_{U_{x_0}}(0)=1/2$ and $\Theta_{U_{x_n}}(0)\geq 1/2+\delta$. Since the function $\varphi_n(r)=W(U_{x_n},r)+Cr^{\delta_{\text{\tiny\sc A}}}$ is non-decreasing by Proposition \ref{p:weissder}, we have for every $r>0$
\begin{equation*}
\frac12+\delta\leq\Theta_{U_{x_n}}(0)=\frac{1}{\Lambda\omega_d}\lim_{s\rightarrow 0^+}W(U_{x_n},s) = \frac{1}{\Lambda\omega_d}\lim_{s\rightarrow 0^+}\varphi_n(s)\leq\frac{1}{\Lambda\omega_d}\varphi_n(r).
\end{equation*}
Passing to the limit as $n\rightarrow\infty$ and using that $\lim_{n\rightarrow\infty}W(U_{x_n},r)=W(U_{x_0},r)$, it follows that for every $r>0$
\begin{equation*}
\frac12+\delta\leq\frac{1}{\Lambda\omega_d}\lim_{n\rightarrow\infty}\varphi_n(r)=\frac{1}{\Lambda\omega_d}\Big[W(U_{x_0},r)+Cr^{\delta_{\text{\tiny\sc A}}}\Big].
\end{equation*}
But the right hand side converges to $\Theta_{U_{x_0}}(0)=1/2$ as $r\rightarrow 0$ which is a contradiction
\end{proof}

\subsection{The regular part is Reifenberg flat}
We prove that the regular part of $\O^\ast_1$ is locally Reifenberg flat. Recall that by Proposition \ref{p:Reg_is_open}, $\text{Reg}(\partial\O^\ast_1\cap D)$ is relatively open in $\partial\O^\ast_1$. Roughly speaking, a domain is said to be Reifenberg flat if its boundary can be well approximated by hyperplanes. We give here a precise definition.

\begin{definition}
Let $\O\subset\R^d$ be an open set and let $\delta, R>0$. We say that $\O$ is a $(\delta,R)$-Reifenberg flat domain if:
\begin{enumerate}
    \item For every $x\in\partial\O$ there exist an hyperplane $H=H_{x,R}$ containing $x$ and a unit vector $\nu=\nu_{x,R}\in\partial B_1\subset\R^d$ orthogonal to $H$ such that
    \begin{align*}
        &\{y+t\nu\in B_R(x) \, : \, y\in H, \, t\geq 2\delta R\} \subset \O, \\
        &\{y-t\nu\in B_R(x) \, : \, y\in H, \, t\geq 2\delta R\} \subset \R^d\backslash\O.
    \end{align*}
    \item  For every $x\in\partial \O$ and every $r\in(0,R]$ there exists an hyperplane $H=H_{x,r}$ containing $x$ such that
    \begin{equation*}
        \text{dist}_{\mathcal{H}}(\partial\O\cap B_r(x),H\cap B_r(x))<\delta r.
    \end{equation*}
\end{enumerate}
\end{definition}

\begin{prop}\label{p:reifenberg}
Let $\delta>0$. Then, for every $x_0\in\text{Reg}(\partial\O^\ast_1\cap D)$ there exists $R=R(x_0)>0$ such that $\O^\ast_1$ is $(\delta,R)$-Reifenberg flat in a neighborhood of $x_0$.
\end{prop}

\begin{proof}
Assume by contradiction that there exists $\delta>0$ and $x_0\in\text{Reg}(\partial\O^\ast_1\cap D)$ such that, for every $R>0$, $\O^\ast_1$ is not $(\delta,R)$-Reifenberg flat in any neighborhood of $x_0$. Then, there exist sequences $x_n\rightarrow x_0$, $x_n\in\partial\O^\ast_1$, and $r_n\downarrow 0$ such that one of the following assertion holds
\begin{enumerate}[label=\roman*)]
    \item For every hyperplane $H$ containing $x_n$ and every $\nu\in\partial B_1$ we have either
    \begin{equation*}
    \{y+t\nu\in B_{r_n}(x_n) \, : \, y\in H, \, t\geq 2\delta r_n\} \nsubseteq \O^\ast_1 \quad\text{or}\quad\{y-t\nu\in B_{r_n}(x_n) \, : \, y\in H, \, t\geq 2\delta r_n\} \nsubseteq \R^d\backslash\O^\ast_1.
    \end{equation*}
    \item For every hyperplane containing $x_n$ we have
    \begin{equation*}
        \text{dist}_{\mathcal{H}}(\partial\O^\ast_1\cap B_{r_n}(x_n),H\cap B_{r_n}(x_n))\geq \delta r_n.
    \end{equation*}
\end{enumerate}
We consider the blow-up sequence $B_n(\xi)=\frac{1}{r_n}U(x_n+r_n\xi)$ and set $\O_n=\{|B_n|>0\}$. Then the above assumptions can be equivalently reformulated as
\begin{enumerate}[label=\roman*')]
    \item For every hyperplane $H$ containing $0$ and every $\nu\in\partial B_1$ we have either
    \begin{equation*}
    \{y+t\nu\in B_1 \, : \, y\in H, \, t\geq 2\delta \} \nsubseteq \O_n \quad\text{or}\quad\{y-t\nu\in B_1 \, : \, y\in H, \, t\geq 2\delta \} \nsubseteq \R^d\backslash\O_n.
    \end{equation*}
    \item For every hyperplane containing $0$ we have
    \begin{equation*}
        \text{dist}_{\mathcal{H}}(\partial\O_n\cap B_1,H\cap B_1)\geq \delta.
    \end{equation*}
\end{enumerate}
Notice that $x_n\in\text{Reg}(\partial\O^\ast_1\cap D)$ for $n$ large enough since $\text{Reg}(\partial\O^\ast_1\cap D)$ is an open subset of $\partial\O^\ast_1$ (Proposition \ref{p:Reg_is_open}).
Up to a subsequence, $B_n$ and $\tilde{B}_n=B_n\circ A^{\sfrac12}_{x_n}$ converge (in  the sense of Proposition \ref{p:convblowup}) to $B_0$ and $\tilde{B}_0=B_0\circ A^{\sfrac12}_{x_0}$ respectively. 

We first prove that
\begin{equation}\label{e:reifenberg1}
   W(\tilde{B}_0,r) = \frac{\Lambda\omega_d}{2}\quad\text{for every}\quad r>0.
\end{equation}
By Proposition \ref{p:weissder}, $\varphi_n(r):=W(U_{x_n},rr_n)+C(rr_n)^{\delta_{\text{\tiny\sc A}}}$ is a non-decreasing function. Moreover, by Lemma \ref{l:densityUx} and since we have $\Theta_{U_{x_n}}(0)=1/2$ (Lemma \ref{l:carac_reg_part}), it follows that 
\begin{equation*}
    \lim_{r\rightarrow 0^+}\varphi_n(r) = \lim_{r\rightarrow 0^+}W(U_{x_n},r)=\Lambda\omega_d\Theta_{U_{x_n}}(0)=\frac{\Lambda\omega_d}{2}.
\end{equation*}
We now fix $r>0$ and $\varepsilon>0$. Since $\lim_{s\rightarrow 0^+} W(U_{x_0},s)=\Lambda\omega_d\Theta_{U_{x_0}}(0)=\frac{\Lambda\omega_d}{2}$ there exists $\overline{r}>0$ such that
\begin{equation*}
    W(U_{x_0},\overline{r})+C\overline{r}^{\delta_{\text{\tiny\sc A}}} \leq \frac{\Lambda\omega_d}{2} + \varepsilon.
\end{equation*}
Moreover, since $\lim_{n\rightarrow\infty}W(U_{x_n},\overline{r})=W(U_{x_0},\overline{r})$, we have for $n$ large enough that
\begin{equation*}
    W(U_{x_n},\overline{r}) \leq W(U_{x_0},\overline{r})+\varepsilon.
\end{equation*}
Therefore, choosing $n$ large enough so that $rr_n\leq\overline{r}$, we get that
\begin{equation*}
    \frac{\Lambda\omega_d}{2}\leq\varphi_n(r)\leq\varphi_n\bigg(\frac{\overline{r}}{r_n}\bigg)=W(U_{x_n},\overline{r})+C\overline{r}^{\delta_{\text{\tiny\sc A}}}\leq W(U_{x_0},\overline{r})+\varepsilon+C\overline{r}^{\delta_{\text{\tiny\sc A}}}+\frac{\Lambda\omega_d}{2}+2\varepsilon,
\end{equation*}
which proves that
\begin{equation*}
    \lim_{n\rightarrow\infty}\varphi_n(r)=\frac{\Lambda\omega_d}{2}\quad\text{for every}\quad r>0.
\end{equation*}
Since $\tilde{B}_n$ converges strongly in $H^1_{\text{loc}}$ to $\tilde{B}_0$ and $\ind_{\tilde{\O}_n}$ converges in $L^1_{\text{loc}}$ to $\ind_{\tilde{\O}_0}$ we have that $\lim_{n\rightarrow\infty}W(\tilde{B}_n,r)=W(\tilde{B}_0,r)$. Hence we get for every $r>0$
\begin{equation*}
\frac{\Lambda\omega_d}{2}=\lim_{n\rightarrow\infty}\varphi_n(r)=\lim_{n\rightarrow\infty}W(U_{x_n},rr_n)=\lim_{n\rightarrow\infty}W(\tilde{B}_n,r)=W(\tilde{B}_0,r).
\end{equation*}

Now, since $\tilde{B}_0$ is solution of the Alt-Caffarelli functional (Proposition \ref{l:optblowup}) and since $W(\tilde{B}_0,r)$ is constant by \eqref{e:reifenberg1}, it follows from Proposition \ref{p:minAFonehom} that $\tilde{B}_0$ is one-homogeneous, and hence by \eqref{e:densityUx2} that
\begin{equation*}
\frac12=\frac{1}{\Lambda\omega_d}W(\tilde{B}_0,r)=\frac{|\{|\tilde{B}_0|>0 \}\cap B_r|}{|B_r|}.
\end{equation*}
Then, as in the proof of Lemma \ref{l:carac_reg_part}, we get that $\O_0=\{|B_0|>0\}$ is an half-space and hence that $\partial\O_0=\partial\{|B_0|>0\}$ is an hyperplane (containing $0$). This is in contradiction with both assumptions i') and ii') since $\overline{\O}_n$ and $\O_n^c$ converge locally Hausdorff to $\overline{\O}_0$ and $\O_0^c$ respectively (Proposition \ref{p:convblowup}). This concludes the proof.
\end{proof}

\subsection{The regular part is $C^{1,\alpha}$}
We prove that the regular part of $\O^\ast_1$ is $C^{1,\alpha}$-regular and that it is $C^{\infty}$-regular provided that $a_{ij},b\in C^{\infty}$ (see Proposition \ref{p:Cinfty_reg}). Using a boundary Harnak principle for non-tangentially accessible (NTA) domains proved by Jerison and Kenig in \cite{jerison-kenig-82}, we prove that the first eigenfunction satisfies an optimality condition on $\O^\ast_1$. The proof then follows from the regularity result of De Silva for the one-phase free boundaries (see \cite{de-silva-11}). 

\begin{prop}\label{p:Cinfty_reg}
The regular part $\text{Reg}(\partial\O^\ast_1\cap D)$ is locally the graph of a $C^{1,\alpha}$ function. Moreover, if $a_{i,j}\in C^{k,\delta}(D)$ and $b\in C^{k-1,\delta}(D)$, for some $\delta\in(0,1)$ and $k\geq 1$, then $\text{Reg}(\partial\O^\ast_1\cap D)$ is locally the graph of a $C^{k+1,\alpha}$ function. In particular, if $a_{i,j}, b\in C^{\infty}(D)$, then $\text{Reg}(\partial\O^\ast_1\cap D)$ is locally the graph of a $C^{\infty}$ function.
\end{prop}

\begin{definition}
A bounded open set $\O\subset\R^d$ is NTA with constants $M>1$ and $r_0>0$ if the following conditions hold: 
\begin{itemize}
\item (Corkscrew condition) For every $x\in\partial\O$ and $r\in(0,r_0)$ there exists $z_r(x)\in\O$ such that
\begin{equation*}
M^{-1}r<d(z_r(x),\partial\O)<|x-z_r(x)|<r,
\end{equation*}
\item $\R^d\backslash\O$ satisfies the corkscrew condition,
\item (Harnack chain condition) If $\varepsilon>0$, $x_1,x_2\in\O$, $d(x_i,\partial\O)>\varepsilon$, $|x_1-x_2|<k\varepsilon$, then there exists a sequence of $Mk$ overlapping balls included in $\O$ of radius $\varepsilon/M$ such that, the first one is centered at $x_1$ and the last one at $x_2$, and such that the center of two consecutive balls are at most $\varepsilon/(2M)$ apart.
\end{itemize}
\end{definition}

We now recall that any $(\delta,R)$-Reifenberg flat set is NTA, provided that $\delta>0$ is small enough. This result is due to Kenig and Toro, see \cite[Theorem 3.1]{kenig-toro-97}.

\begin{teo}[Reifenberg flat implies NTA]\label{t:reif_implies_nta}
There exists $\delta_0>0$ such that if $\O\subset\R^d$ is a $(\delta,R)$-Reifenberg flat domain for some $R>0$ and some $\delta\leq\delta_0$, then $\O$ is an NTA domain.
\end{teo}

In the following theorem we state the Boundary Harnack Principle for NTA domains and for solutions of uniformly elliptic equations in divergence form with bounded, measurable coefficients. We refer to \cite[Corollary 1.3.7]{kenig-94} or \cite[Lemma 4.10]{jerison-kenig-82} for a proof (see also \cite{de-silva-savin-19} for operator in non-divergence form).

\begin{teo}[Boundary Harnack principle]\label{t:harnack}
Let $\O\subset\R^d$ be an NTA domain and $2r\in(0,r_0)$. Let $\tilde{A}:\R^d\rightarrow\text{Sym}_d^+$ be uniformly elliptic (i.e. $\exists\lambda>0$, $\lambda^{-1}|\xi|^2\leq\xi\cdot \tilde{A}_x\,\xi\leq\lambda|\xi|^2$ $\,\forall x,\xi\in\R^d$) with bounded measurable coefficients. Let $x_0\in\partial\O$ and let $u,v\in H^1(\O\cap B_{2r}(x_0))\cap C(\O\cap B_{2r}(x_0))$ be such that $u,v=0$ on $\partial\O\cap B_{2r}(x_0)$, $v>0$ in $\O\cap B_{2r}(x_0)$ and 
\begin{equation*}
\dive(\tilde{A}\nabla u)=\dive(\tilde{A}\nabla v)=0\quad\text{in}\quad \O\cap B_{2r}(x_0).
\end{equation*}
Then there exists $C>0$, depending only on $d$ and $\lambda$ and the NTA constants, such that
\begin{equation}\label{e:harnacka}
C^{-1}\frac{u(z_r(x_0))}{v(z_r(x_0))} \leq \frac{u(x)}{v(x)} \leq C\frac{u(z_r(x_0))}{v(z_r(x_0))}\quad\text{for every}\quad x\in\O\cap B_r(x_0).
\end{equation}
\end{teo}

Since the estimate \eqref{e:harnacka} holds for every harmonic functions with a uniform constant, it is standard to deduce that the quotient of two harmonics functions on an NTA domain is H\"{o}lder continuous up to the boundary.
We refer to \cite[Corollary 1.3.9]{kenig-94} or \cite[Theorem 7.9]{jerison-kenig-82} (see also \cite[Corollary 1]{athanasopoulos-caffarelli-85}).

\begin{coro}\label{c:harnack_holder}
Let $\O, \tilde{A}, x_0, r$ and $u,v$ be as in Theorem \ref{t:harnack}. Then there exist constants $\alpha\in(0,1)$ and $C>0$, depending only on $d$ and $\lambda$ and the NTA constants, such that
\begin{equation*}
\left|\frac{u(x)}{v(x)}-\frac{u(y)}{v(y)}\right|\leq C\frac{u(z_r(x_0))}{v(z_r(x_0))}\left(\frac{|x-y|}{r}\right)^\alpha\quad\text{for every}\quad x,y\in\O\cap B_r(x_0).
\end{equation*}
In particular, for every $x\in\partial\O\cap B_r(x_0)$ the limit $\lim_{\O\ni y\rightarrow x}\frac{u(y)}{v(y)}$ exists and  $\frac{u}{v}:\overline{\O}\cap B_r(x_0)\rightarrow\R$ is $\alpha$-H\"{o}lder continuous.
\end{coro}


We now prove the analogous boundary Harnack theorem for the eigenfunctions on an optimal set $\O^\ast$ to the problem \eqref{e:shapeopt}. We notice that in the proof it is essential that the first eigenfunction $u_1$ is positive and non-degenerate (Proposition \ref{p:nondegu1}). The case of the eigenfunctions for the Laplacian is already treated in \cite[Appendix A]{ramos-tavares-terracini-16}. We extend this result to the case of the operator $-b^{-1}\dive(A\nabla\cdot)$. We highlight that one of the difficulty comes from the presence of the Lipschitz function $b$. 

\begin{teo}[Boundary Harnack principle for eigenvalues]\label{t:boundary_harnack_eig}
Let $U=(u_1,\dots,u_k)$ be the first $k$ normalized eigenfunctions on $\O^\ast$ and let $x_0\in\text{Reg}(\partial\O^\ast_1\cap D)$. Then $\O^\ast_1$ is NTA in $B_r(x_0)$ for some $r=r(x_0)>0$ and there exists $\alpha\in(0,1)$, depending only on $d, \lambda_{\text{\tiny\sc A}}$ and the NTA constants of $\O^\ast_1$, such that for every $i=2,\dots,k$
\begin{equation*}
\frac{u_i}{u_1} \text{ is } \alpha\text{-H\"{o}lder continuous in } \overline{\O^\ast_1}\cap B_r(x_0).
\end{equation*}
\end{teo}

We will need the following Lemma.

\begin{lm}\label{l:ext_positive_part}
Let $\O\subset D$ be a quasi-open set, $u\in H^1_0(\O)$ and $\lambda>0$. Then, for every $x_0\in\partial\O\cap D$ there exists $r_0>0$ depending only on $d, \lambda_{\text{\tiny\sc A}}, c_b$ and $\lambda$, such that for every $r\leq r_0$ with $B_r(x_0)\subset D$, there exists a unique solution $v\in H^1_0(D)$ of
\begin{equation}\label{e:ext_positive_parta}
\left\lbrace
         \begin{array}{ll}
              -\dive(A\nabla v)=\lambda bv\quad &\text{in}\quad\O\cap B_r(x_0)\\
              v=u,\quad&\text{on}\quad\partial(\O\cap B_r(x_0)).\\
         \end{array}
\right. 
\end{equation}
If, moreover, $u\in L^\infty(D)$, then $v\in L^\infty(D)$ and we have the estimate 
\begin{equation}\label{e:ext_positive_partb}
\|v\|_{L^\infty(\O\cap B_r(x_0))}\leq C\big(r\|u\|_{H^1(\O;m)} + \|u\|_{L^\infty(\partial B_r(x_0))}\big)
\end{equation}
where the constant $C>0$ depends only on $d, \lambda_{\text{\tiny\sc A}}, c_b$ and $\lambda$.
\end{lm}

\begin{proof}
Observe that any minimizer in $\mathcal{A}:=\{\varphi\in H^1_0(D): u-\varphi\in H^1_0(\O\cap B_r(x_0))\}$ of the functional
\begin{equation*}
\tilde{J}(\varphi)=\int_DA\nabla\varphi\cdot\nabla\varphi - \lambda\int_D\varphi^2b
\end{equation*}
is solution of \eqref{e:ext_positive_parta}. Therefore, it is enough to prove that $\{\varphi\in\mathcal{A}\ : \ \tilde{J}(\varphi)\leq C\}$ is weakly compact in $H^1_0(D)$ to prove the existence of a function $v$ solution of \eqref{e:ext_positive_parta}. We first compute
\begin{align*}
\int_D\varphi^2b &\leq 2\int_{\O\cap B_r(x_0)}(\varphi-u)^2b+2\int_Du^2b \\
&\leq \frac{2}{\lambda_1(\O\cap B_r(x_0))}\int_{\O\cap B_r(x_0)}A\nabla(\varphi-u)\cdot\nabla(\varphi-u) + 2\int_Du^2b \\
&\leq \frac{4\lambda_{\text{\tiny\sc A}}^2}{\lambda_1(B_{r_0}(x_0))}\int_D\big(|\nabla\varphi|^2+|\nabla u|^2\big) + 2\int_Du^2b.
\end{align*}
Then, for $r_0$ small enough (such that $4\lambda_{\text{\tiny\sc A}}^4\lambda\leq\lambda_1^{-\Delta}(B_1)/(2\lambda_{\text{\tiny\sc A}}^2c_br_0^2)$ where $\lambda_1^{-\Delta}(B_1)$ stands for the first eigenvalue of the Dirichlet Laplacian on $B_1$) we have
\begin{equation*}
\int_D|\nabla\varphi|^2 \leq \lambda_{\text{\tiny\sc A}}^2\tilde{J}(\varphi)+\lambda_{\text{\tiny\sc A}}^2\lambda\int_D\varphi^2b\leq \lambda_{\text{\tiny\sc A}}^2\tilde{J}(\varphi)+ \frac12\int_D\big(|\nabla\varphi|^2+|\nabla u|^2\big) +2\lambda_{\text{\tiny\sc A}}^2\lambda\int_Du^2b,
\end{equation*}
which gives that
\begin{equation*}
\int_D|\nabla\varphi|^2 \leq 2\lambda_{\text{\tiny\sc A}}^2\tilde{J}(\varphi)+ \int_D|\nabla u|^2 + 4\lambda_{\text{\tiny\sc A}}^2\lambda\int_Du^2b\leq 2\lambda_{\text{\tiny\sc A}}^2\tilde{J}(\varphi)+(1+4\lambda_{\text{\tiny\sc A}}^2\lambda)\|u\|_{H^1(\O;m)}.
\end{equation*}
This proves the existence of $v$, and the uniqueness easily follows provided that $\lambda\leq\lambda_1(\O\cap B_r(x_0))$.

We now prove the $L^\infty$-estimate. We consider the functions defined by
\begin{equation*}
\left\lbrace
         \begin{array}{ll}
              \dive(A\nabla h)=0,\qquad -\dive(A\nabla w)=\lambda bv&\text{in}\quad\O\cap B_r(x_0)\\
              h=u,\qquad\qquad\qquad w=0&\text{on}\quad\partial(\O\cap B_r(x_0)).\\
         \end{array}
\right. 
\end{equation*}
Reasoning as above this functions exist and are unique, and we have $v=h+w$. Let $R=R_{\O\cap B_r(x_0)}$ be the resolvent of $-b^{-1}\dive(A\nabla\cdot)$ in $\O\cap B_r(x_0)$. We have the estimates $\|R\|_{\mathcal{L}(L^2,L^{2^\ast})}\leq C_d$ where $2^\ast=\frac{2d}{d-2}$ and $\|R\|_{\mathcal{L}(L^d,L^\infty)}\leq Cr$ by \cite[Lemma 2.1]{trey-19}, where the constant $C$ depends only on $d, \lambda_{\text{\tiny\sc A}}$ and $c_b$. Notice also that we have
\begin{equation*}
v=\lambda^nR^n(v)+\sum_{i=0}^{n-1}\lambda^iR^i(h),
\end{equation*}
and that $\|h\|_{L^\infty(\O\cap B_r(x_0))}\leq \|u\|_{L^\infty(\partial B_r(x_0))}$ by the maximum principle.
Therefore, with an interpolation argument, there exists a dimensional constant $n\geq 1$ such that we have the estimate
\begin{equation*}
\|v\|_{L^\infty(\O\cap B_r(x_0))}\leq C\big(r\|v\|_{L^2(D;m)}+\|u\|_{L^\infty(\partial B_r(x_0))}\big),
\end{equation*}
where now $C$ also depends on $\lambda$. Hence, it remains only to estimate $\|v\|_{L^2(D;m)}$ to complete the proof.
Then, for $r_0$ small enough, we have
\begin{align*}
\int_Dv^2b &\leq 2\int_{\O\cap B_r(x_0)}(v-u)^2b+2\int_Du^2b \leq \frac{4}{\lambda_1(B_{r_0}(x_0))}\int_D\big(|\nabla v|^2 + |\nabla u|^2 \big)+2\int_Du^2b \\
&\leq \frac{1}{2\lambda}\tilde{J}(v) +\frac12\int_Dv^2b + \frac{1}{2\lambda_{\text{\tiny\sc A}}^2\lambda}\int_D|\nabla u|^2 + 2\int_Du^2b,
\end{align*}
which implies that (since $\tilde{J}(v)\leq\tilde{J}(u)$)
\begin{equation*}
\int_Dv^2b\leq \frac{1}{\lambda}\tilde{J}(u) +\frac{1}{\lambda_{\text{\tiny\sc A}}^2\lambda}\int_D|\nabla u|^2 + 4\int_Du^2b \leq \Big(\frac{2\lambda_{\text{\tiny\sc A}}^2}{\lambda}+4\Big)\|u\|_{H^1(D;m)}.
\end{equation*}
\end{proof}

\begin{proof}[Proof of Theorem \ref{t:boundary_harnack_eig}]
By Proposition \ref{p:reifenberg} and Theorem \ref{t:reif_implies_nta}, $\O^\ast_1$ is an NTA domain near $x_0$. Let $\alpha$ be the constant given by Corollary \ref{c:harnack_holder} and set $\beta=\frac{\alpha}{1+\alpha}$. Let $x,y\in B_r(x_0)$ and set $\overline{r}=|x-y|^\beta$, $d_x=d(x,\partial\O^\ast_1)$, $d_y=d(y,\partial\O^\ast_1)$. 
We divide the proof in three steps.

{\it Step 1.} Assume that $d_x, d_y \geq 2\overline{r}$. By a change of variables, it follows that $\tilde{u}(z)=\overline{r}^{-1}u_1(x+\overline{r}z)$ is solution of
\begin{equation*}
-\dive(\tilde{A}\nabla\tilde{u})=\overline{r}^2\lambda_1(\O^\ast)\tilde{b}\tilde{u}\quad\text{in}\quad B_2,
\end{equation*}
where we have set $\tilde{A}_z=A_{x+\overline{r}z}$ and $\tilde{b}(z)=b(x+\overline{r}z)$. By standard Schauder estimates (see \cite[Theorem 8.8]{gilbarg-trudinger-01}) we have 
\begin{equation*}
\|\tilde{u}\|_{C^{1,\delta_{\text{\tiny\sc A}}}(B_{1})}\leq C\big(\|\tilde{u}\|_{L^\infty(B_2)} + \overline{r}^2\lambda_1(\O^\ast)\|\tilde{b}\tilde{u}\|_{L^\infty(B_2)}\big),
\end{equation*}
where $C$ depends only on $d, c_{\text{\tiny\sc A}}$ and $\lambda_{\text{\tiny\sc A}}$. In particular,
\begin{equation*}
\|u_1\|_{C^1(B_{\overline{r}}(x))} \leq \|\tilde{u}\|_{C^1(B_{1})}\leq \|\tilde{u}\|_{C^{1,\delta_{\text{\tiny\sc A}}}(B_{1})} \leq C\|\tilde{u}\|_{L^\infty(B_2)} \leq \frac{C}{\overline{r}}.
\end{equation*}
Similarly, we have $\|u_i\|_{C^1(B_{\overline{r}}(x))} \leq C/\overline{r}$. On the other hand, by non-degeneracy of $u_1$ we have $u_1(x)\geq cd_x$ and $u_1(y)\geq cd_y$ for some constant $c>0$. Therefore, since $u_i$ is $L$-Lipschitz continuous (and because $y\in B_{\overline{r}}(x)$), we get
\begin{align*}
\bigg|\frac{u_i(x)}{u_1(x)}-\frac{u_i(y)}{u_1(y)}\bigg| &\leq \frac{|u_i(x)-u_i(y)|}{u_1(x)}+\frac{|u_1(x)-u_1(y)|\,|u_i(y)|}{u_1(x)u_1(y)}\leq \frac{C}{\overline{r}}|x-y|\Big(\frac{1}{cd_x} +\frac{Ld_y}{c^2d_xd_y}\Big) \\
&\leq \frac{C}{\overline{r}^2}|x-y| = C|x-y|^{1-2\beta}\leq C|x-y|^\beta,
\end{align*}
where the last inequality holds provided that $\beta\leq1/3$.

{\it Step 2.} Assume that $d_x\leq 2\overline{r}$. Let $\overline{x}\in\partial\O^\ast_1$ such that $d_x=|\overline{x}-x|$. 
We write for simplicity $\lambda_1=\lambda_1(\O^\ast)$, $\lambda_i=\lambda_i(\O^\ast)$ and $B=B_{6\overline{r}}(\overline{x})$. 
Since $u_i$ may change its sign, we consider the functions 
\begin{equation*}
\left\lbrace
         \begin{array}{ll}
              -\dive(A\nabla v_i)=\lambda_ibv_i,\qquad -\dive(A\nabla w_i)=\lambda_ibw_i&\text{in}\quad\O^\ast_1\cap B\\
              v_i=u_i^+,\qquad\qquad\qquad\quad \ \ w_i=u_i^-&\text{on}\quad\partial(\O^\ast_1\cap B).\\
         \end{array}
\right. 
\end{equation*}
These functions exist thanks to Lemma \ref{l:ext_positive_part} and we have $u_i=v_i-w_i$. 
We now set $m=\min_{z\in B}b(z)$ and $M=\max_{z\in B}b(z)$ and $I=(-1,1)$. Moreover, for $(z,z_{d+1})\in (\O^\ast_1\cap B)\times I$ we define the functions
\begin{align*}
&u_{1,m}(z,z_{d+1})=e^{-\sqrt{\lambda_1m}z_{d+1}}u_1(z)\qquad&u_{1,M}(z,z_{d+1})=e^{-\sqrt{\lambda_1M}z_{d+1}}u_1(z) \\
&u_{i,m}(z,z_{d+1})=e^{-\sqrt{\lambda_im}z_{d+1}}v_i(z)\qquad&u_{i,M}(z,z_{d+1})=e^{-\sqrt{\lambda_iM}z_{d+1}}v_i(z).
\end{align*}
We define the matrix-valued function $\tilde{A} : (\O^\ast_1\cap B)\times I\subset\R^{d+1}\rightarrow\text{Sym}_{d+1}^+$ by
\begin{equation*}
\tilde{A}_{(z,z_{d+1})}=\begin{pmatrix} A_z & 0 \\ 0 & 1 \end{pmatrix}
\qquad\text{for every}\quad (z,z_{d+1})\in (\O^\ast_1\cap B)\times I.
\end{equation*}
Moreover, we define the harmonic extensions of the above functions as follows
\begin{equation*}
\left\lbrace
         \begin{array}{ll}
              \dive(\tilde{A}\nabla h_{1,m})=\dive(\tilde{A}\nabla h_{1,M})=\dive(\tilde{A}\nabla h_{i,m})=\dive(\tilde{A}\nabla h_{i,M})=0&\text{in}\quad(\O^\ast_1\cap B)\times I\\
              h_{1,m}=u_{1,m}, \quad h_{1,M}=u_{1,M}, \quad h_{i,m}=u_{i,m}, \quad h_{i,M}=u_{i,M}&\text{on}\quad\partial\big[(\O^\ast_1\cap B)\times I\big]\\
         \end{array}
\right. 
\end{equation*}
Now, we get with an easy computation that
\begin{equation*}
\dive(\tilde{A}\nabla(u_{1,m}-h_{1,m}))=\lambda_1e^{-\sqrt{\lambda_1m}x_{d+1}}(m-b(x))u_1(x)\leq 0\quad\text{in}\quad(\O^\ast_1\cap B)\times I,
\end{equation*}
which, by the weak maximum principle, implies that $h_{1,m}\leq u_{1,m}$ in $(\O^\ast_1\cap B)\times I$. Similarly we have (since the functions $u_{i,m}, u_{i,M}$ are positive)
\begin{equation}\label{e:boundary_harnack_eig1}
h_{1,m}\leq u_{1,m},\quad u_{1,M}\leq h_{1,M},\quad h_{i,m}\leq u_{i,m},\quad u_{i,M}\leq h_{i,M}\quad\text{in}\quad (\O^\ast_1\cap B)\times I.
\end{equation}
Moreover, using again the maximum principle, we have the following inequalities
\begin{equation}\label{e:boundary_harnack_eig2}
\frac{e^{\sqrt{\lambda_1m}}}{e^{\sqrt{\lambda_1M}}}\,h_{1,m}\leq h_{1,M}\leq \frac{e^{\sqrt{\lambda_1M}}}{e^{\sqrt{\lambda_1m}}}\,h_{1,m},\quad
\text{in}\quad (\O^\ast_1\cap B)\times I,
\end{equation}
and similarly we have
\begin{equation}\label{e:boundary_harnack_eig3}
\frac{e^{\sqrt{\lambda_im}}}{e^{\sqrt{\lambda_iM}}}\,h_{i,m}\leq h_{i,M}\leq \frac{e^{\sqrt{\lambda_iM}}}{e^{\sqrt{\lambda_im}}}\,h_{i,m},\quad
\text{in}\quad (\O^\ast_1\cap B)\times I.
\end{equation}
Now, since $x,y\in B_{3\overline{r}}(\overline{x})\subset B$, we can use \eqref{e:boundary_harnack_eig1}, \eqref{e:boundary_harnack_eig2} and \eqref{e:boundary_harnack_eig3} to estimate
\begin{align*}
\frac{v_i(x)}{u_1(x)}-\frac{v_i(y)}{u_1(y)} & = \frac{u_{i,M}(x,0)}{u_{1,m}(x,0)}-\frac{u_{i,m}(y,0)}{u_{1,M}(y,0)} \leq \frac{h_{i,M}(x,0)}{h_{1,m}(x,0)}-\frac{h_{i,m}(y,0)}{h_{1,M}(y,0)} \\
&\leq \frac{e^{\sqrt{\lambda_iM}}}{e^{\sqrt{\lambda_im}}}\frac{h_{i,m}(x,0)}{h_{1,m}(x,0)}
-\frac{e^{\sqrt{\lambda_1m}}}{e^{\sqrt{\lambda_1M}}}\frac{h_{i,m}(y,0)}{h_{1,m}(y,0)} \\
&\leq \frac{h_{i,m}(x,0)}{h_{1,m}(x,0)}
-\frac{h_{i,m}(y,0)}{h_{1,m}(y,0)} +C\overline{r}\frac{h_{i,m}(x,0)}{h_{1,m}(x,0)} +C\overline{r}\frac{h_{i,m}(y,0)}{h_{1,m}(y,0)} 
\end{align*}
where the last inequality follows from the definitions of $m,M$ and the fact that $b$ is a Lipschitz continuous function. Now, observe that $\O^\ast_1\times I\subset\R^{d+1}$ is an NTA domain near $(\overline{x},0)$ with the same constants than $\O^\ast_1$. By Corollary \ref{c:harnack_holder}, setting $z_0=z_{3\overline{r}}(\overline{x},0)\in\R^{d+1}$, we have (notice also that $x,y\in B_{3\overline{r}}(\overline{x})$)
\begin{align*}
\frac{h_{i,m}(x,0)}{h_{1,m}(x,0)}
-\frac{h_{i,m}(y,0)}{h_{1,m}(y,0)} &\leq C\frac{h_{i,m}(z_0)}{h_{1,m}(z_0)} \left(\frac{|x-y|}{3\overline{r}}\right)^\alpha = C\overline{r}\frac{h_{i,m}(z_0)}{h_{1,m}(z_0)},
\end{align*}
where in the last equality we have used that $\overline{r}=|x-y|^\beta$ with $\beta=\frac{\alpha}{1+\alpha}$. Moreover, by Theorem \ref{t:harnack} we have
\begin{equation*}
\frac{h_{i,m}(x,0)}{h_{1,m}(x,0)}\leq C\frac{h_{i,m}(z_0)}{h_{1,m}(z_0)},\qquad \frac{h_{i,m}(y,0)}{h_{1,m}(y,0)}\leq C\frac{h_{i,m}(z_0)}{h_{1,m}(z_0)},
\end{equation*}
which finally gives
\begin{equation}\label{e:boundary_harnack_eig4}
\frac{v_i(x)}{u_1(x)}-\frac{v_i(y)}{u_1(y)} \leq C\overline{r}\frac{h_{i,m}(z_0)}{h_{1,m}(z_0)}.
\end{equation}
Then, using \eqref{e:boundary_harnack_eig2} and \eqref{e:boundary_harnack_eig1} we have the following estimate
\begin{equation*}
h_{1,m}(z_0)\geq \frac{e^{\sqrt{\lambda_1m}}}{e^{\sqrt{\lambda_1M}}}h_{1,M}(z_0) \geq \frac{e^{\sqrt{\lambda_1m}}}{e^{\sqrt{\lambda_1M}}}u_{1,M}(z_0) \geq \left(\frac{e^{\sqrt{\lambda_1m}}}{e^{\sqrt{\lambda_1M}}}\right)^2u_{1,m}(z_0)\geq Cu_{1,m}(z_0).
\end{equation*}
Now, in view of the definition of $z_0=z_{3\overline{r}}(\overline{x},0)\in\R^{d+1}$ we have $d(z_0, \partial(\O^\ast_1\times I))>3\overline{r}M^{-1}$ and by non-degeneracy of $u_1$ (Proposition \ref{p:nondegu1}) it follows that $u_{1,m}(z_0)\geq C\overline{r}$. Moreover, by \eqref{e:ext_positive_partb}, it follows that $\|v_i\|_{L^\infty(B)}\leq C\overline{r}$ since $u_i$ is Lipschitz continuous. Therefore we have
\begin{equation*}
\frac{v_i(x)}{u_1(x)}-\frac{v_i(y)}{u_1(y)} \leq C\overline{r}\frac{u_{i,m}(z_0)}{u_{1,m}(z_0)}\leq C\overline{r}\frac{\|v_i\|_{L^\infty(B)}}{u_{1,m}(z_0)} \leq C\overline{r}=C|x-y|^\beta.
\end{equation*}
This concludes the proof since the same estimate also holds for $w_i$ and that we have $u_i/u_1=v_i/u_1-w_i/u_1$.
\end{proof}

As a consequence of the optimality condition of $U$ (Lemma \ref{l:opt_cond_U}) and of the boundary Harnack principle (Theorem \ref{t:boundary_harnack_eig}), it follows that the first eigenfunction is solution of a one-phase free boundary problem on $\O^\ast_1$.

\begin{lm}\label{l:opt_cond_u1}
For every $x_0\in\text{Reg}(\partial\O^\ast_1\cap D)$ there exist $r=r(x_0)>0$, $c\in(0,1)$ and a H\"{o}lder continuous function $g:\partial\O^\ast_1\cap B_r(x_0)\rightarrow[c,1]$ such that $u_1$ is a viscosity solution to the problem
\begin{equation*}
-\dive(A\nabla u_1)=\lambda_1(\O^\ast)bu_1\ \ \text{in}\ \ \O^\ast_1,\quad u_1=0\ \ \text{on}\ \ \partial\O^\ast_1,\quad |A^{\sfrac12}[\nabla u_1]|=g\sqrt{\Lambda}\ \ \text{on}\ \ \partial\O^\ast_1\cap B_r(x_0).
\end{equation*}
\end{lm}

\begin{proof}
Let $U=(u_1,\dots,u_k)$ be the first $k$ eigenfunctions on $\O^\ast$. By Theorem \ref{t:boundary_harnack_eig} the functions $g_i:=\frac{u_i}{u_1}:\partial\O^\ast\cap B_r(x_0)\rightarrow\R$, for $i=2,\dots,k$, are H\"{o}lder continuous. Therefore, the function $g:\partial\O^\ast_1\cap B_r(x_0)\rightarrow[0,1]$ defined by
\begin{equation}\label{e:def_g}
g=\frac{1}{\scriptstyle{\sqrt{\mathlarger{1+g_2^2+\cdots+g_k^2}}}}.  
\end{equation}
is also H\"{o}lder continuous. Since $u_1=g|U|$, it follows from the non-degeneracy of $u_1$ that $g\geq c:=C_1^{-1}$ where $C_1$ is the constant from Proposition \ref{p:nondegu1}. Now, let $y\in\partial\O^\ast_1\cap B_r(x_0)$ and let $\varphi\in C^2(D)$ be a function touching $u_1$ by below at the point $y$. Since $1/g$ is $\beta$-H\"{o}lder continuous for some $\beta\in(0,1)$, there exists $C>0$ such that for $\rho>0$ small enough we have
\begin{equation*}
\frac{1}{g(x)}\geq\frac{1}{g(y)}-C|x-y|^\beta\geq 0\quad\text{for every}\quad x\in\overline{\O^\ast_1}\cap B_{\rho}(y).
\end{equation*}
Therefore, the function $\psi(x)=\varphi(x)\Big(\frac{1}{g(y)}-C|x-y|^\beta\Big)$ is such that $\psi(y)=|U(y)|$ and satisfies 
\begin{equation}
\psi(x)\leq u_1(x)\bigg(\frac{1}{g(x_0)}-C|x-y|^\beta\bigg)\leq |U(x)|\quad\text{for every}\quad x\in\overline{\O^\ast_1}\cap B_{\rho}(y).
\end{equation}
This proves that $\psi$ touches $|U|$ by below at the point $y$. On the other hand, $\psi$ is differentiable at $y$ and we have $\nabla\psi(y)=\frac{1}{g(y)}\nabla\varphi(y)$. Therefore, using that $U$ is a viscosity solution of \eqref{e:viscpb}, it follows that
\begin{equation*}
\sqrt{\Lambda}\geq|A^{\sfrac12}_y[\nabla\psi(y)]|\geq\frac{1}{g(y)}|A_y^{\sfrac12}[\nabla\varphi(y)]|.
\end{equation*}
The case when $\varphi$ touches $u_1$ by above is similar. 
\end{proof}

\begin{teo}[Higher boundary Harnack principle for eigenvalues]\label{t:higher_boundary_harnack_eig}
Let $k\geq 1$ and assume that $\O^\ast_1$ is $C^{k,\alpha}$-regular near $x_0\in\partial\O^\ast_1\cap D$ for some $\alpha\in(0,1)$. If $k\geq 2$, suppose moreover that $a_{ij}, b \in C^{k-1,\alpha}(D)$. Then there exists $r>0$ such that for every $i=2,\dots,k$
\begin{equation*}
\frac{u_i}{u_1} \text{ is of class } C^{k,\alpha} \text{ in } \overline{\O^\ast_1}\cap B_r(x_0).
\end{equation*}
\end{teo}

\begin{proof}
Let $R>0$ such that there exists $\varphi\in H^1_0(B_R(x_0))$ satisfying $\varphi>0$ in $B_R(x_0)$ and solution of the equation
\begin{equation*}
-\dive(A\nabla\varphi)=\lambda_1(\O^\ast)b\varphi\quad\text{in}\quad B_R(x_0).
\end{equation*}
(More precisely, we extend $a_{i,j}$ and $b$ to bounded functions in $\R^d$ with $b\geq c_b$, and we choose $R>0$ such that $\lambda_1(B_R)=\lambda_1(\O^\ast)$). Let $2r<R$ be such that $\O^\ast_1$ is $C^{k,\alpha}$-regular in the ball $B_{2r}(x_0)\subset D$. Then we have
\begin{align*}
\dive\left(\varphi^2A\nabla\Big(\frac{u_1}{\varphi}\Big)\right) &= \dive\big(\varphi A\nabla u_1 - u_1A\nabla\varphi\big) \\
&= \varphi\dive(A\nabla u_1) + \nabla\varphi A\nabla u_1 - \nabla u_1 A\nabla\varphi -u_1\dive(A\nabla\varphi) = 0\ \ \text{in}\ \ \O^\ast_1\cap B_{2r}(x_0),
\end{align*}
and similarly
\begin{equation*}
\dive\left(\varphi^2A\nabla\Big(\frac{u_i}{\varphi}\Big)\right) = (\lambda_1(\O^\ast)-\lambda_i(\O^\ast))bu_i\varphi\quad\text{in}\quad \O^\ast_1\cap B_{2r}(x_0).
\end{equation*}
Now, the proof follows by \cite[Theorem 2.4]{de-silva-savin-14} for $k=1$ and by \cite[Theorem 3.1]{de-silva-savin-14} for $k\geq 2$, which say that $u_1/\varphi, u_i/\varphi \in C^{k,\alpha}(\overline{\O^\ast_1}\cap B_r(x_0))$. 
\end{proof}

\begin{proof}[Proof of Proposition \ref{p:Cinfty_reg}]
We prove the regularity by a finite induction on $l\in\{1,\dots,k\}$. For $l=1$, by \cite[Theorem 1.1]{de-silva-11} and Lemma \ref{l:opt_cond_u1} it follows that $\text{Reg}(\partial\O^\ast_1\cap D)$ is locally $C^{1,\alpha}$-regular. Now, if $\text{Reg}(\partial\O^\ast_1\cap D)$ is $C^{l,\alpha}$-regular, $l\leq k$, by Theorem \ref{t:higher_boundary_harnack_eig} and the definition of $g$ in \eqref{e:def_g}, we have that $g$ is a $C^{l,\alpha}$ function on $\text{Reg}(\partial\O^\ast_1\cap D)$. Therefore, in view of Lemma \ref{l:opt_cond_u1} and by \cite[Theorem 2]{kinderlehrer-nirenberg-77} it follows that $\text{Reg}(\partial\O^\ast_1\cap D)$ is locally $C^{l+1,\alpha}$-regular. This completes the proof.
\end{proof}

\subsection{Dimension of the singular set}\label{sub:est_dim}
We prove in this last subsection some kind of smallness of the singular set. We recall that $\O^\ast$ denotes an optimal set to \eqref{e:shapeopt} and that $\O^\ast_1$ stands for any connected component of $\O^\ast$ at which the first eigenfunction is positive. 

An estimate of the dimension of the singular set can be obtain as a consequence of the Federer's Theorem. Indeed, since $\O^\ast_1$ is a set of finite perimeter (Proposition \ref{p:finitperi}) and in view of the density estimate (Proposition \ref{p:density_est}), it follows from the Federer's Theorem (see, for instance, \cite[Theorem 16.2]{maggi-12}) that $\mathcal{H}^{d-1}(\text{Sing}(\partial\O^\ast_1\cap D))=0$. In Proposition \ref{p:dim_sing} below we provide a more precise estimate of the dimension of the singular set.

\begin{definition}\label{d:dstar}
We define $d^\ast$ as the smallest dimension which admits a one-homogeneous global
minimizer of the Alt-Caffarelli functional with exactly one singularity at zero.
\end{definition}

The exact value of the critical dimension $d^\ast$ is still unknown but we know that $d^*\in\{5,6,7\}$ (see \cite{jerison-savin-15} for $d^\ast\geq 5$ and \cite{de-silva-jerison-09} for $d^\ast\leq 7$). The following result on the smallness of the singular set is standard and was first proved in the framework of the minimal surfaces (for which the critical dimension is exactly $8$, see for example \cite[Chapter 11]{giusti-84}). Later, in \cite{weiss-99}, Weiss adapted this strategy for minimizers of the Alt-Caffarelli functional by introducing a monotonicity formula. In \cite{mazzoleni-terracini-velichkov-17}, the authors prove this result in the vectorial setting. In this section we follow the same approach and we extend this result to the case of variable coefficients.

\begin{prop}[On the dimension of the singular set]\label{p:dim_sing}
The singular part $\text{Sing}(\partial\O^\ast_1\cap D)$ is:
\begin{enumerate}
\item empty if $d<d^\ast$,
\item a discrete (locally finite) set if $d=d^\ast$,
\item of Hausdorff dimension at most $(d-d^\ast)$ if $d>d^\ast$, that is, $\mathcal{H}^{d-d^\ast+s}(\text{Sing}(\partial\O^\ast_1\cap D))=0$ for every $s>0$.
\end{enumerate}
\end{prop}

We first prove two preliminary Lemmas and to this aim we extend the definition of the Weiss functional for any ball. 
Let $U\in H^1(\R^d,\R^k)$, $x\in\R^d$ and $r>0$. We set
\begin{equation*}
J(U,x,r)=\int_{B_r(x)}|\nabla U|^2 + \Lambda|\{|U|>0\}\cap B_r(x)|
\end{equation*}
and
\begin{equation*}
W(U,x,r)= \frac{1}{r^d}J(U,x,r)-\frac{1}{r^{d+1}}\int_{\partial B_r(x)}|U|^2.
\end{equation*}
Obviously we have $J(U,r)=J(U,0,r)$ and $W(U,r)=W(U,0,r)$.

\begin{lm}\label{l:comp_weiss}
Let $(x_n)_{n\in\N}\subset\partial\O^\ast_1\cap D$ be a sequence converging to $x_0\in \partial\O^\ast_1\cap D$ and let $B_n=B_{x_0,r_n}$ be a blow-up sequence with fixed center. We set $\tilde{B}_n=B_n\circ A^{\sfrac12}_{x_0}$ and $\tilde{\O}_n=\{|\tilde{B}_n|>0\}$. Then, up to a subsequence, the sequence $y_n=A^{-\sfrac12}_{x_0}\Big[\frac{x_n-x_0}{r_n}\Big]\in\partial\tilde{\O}_n$ converges to some  $y_0$ and, for every small $r>0$, there exists $n_0$ such that for every $n\geq n_0$ we have
\begin{equation}\label{e:comp_weissa}
W(U_{x_n},rr_n)\leq W(\tilde{B}_n,y_0,r) + C|x_0-x_n|^{\delta_{\text{\tiny\sc A}}/2}+C\frac{|y_0-y_n|}{r},
\end{equation}
where the constant $C>0$ depends only on $d, c_{\text{\tiny\sc A}}, \lambda_{\text{\tiny\sc A}}, \Lambda$ and the Lipschitz constant $L=\|\nabla U\|_{L^\infty(K)}$ of $U$ in some compact neighborhood $K\subset D$ of $x_0$.
\end{lm}

\begin{proof}
We first compare $J(U_{x_n},rr_n)$ and $J(\tilde{B}_n,y_0,r)$. Since $U_{x_n}=U\circ F_{x_n}$ by definition, we compute
\begin{align*}
J(U_{x_n},rr_n) &= \int_{B_{rr_n}}\big(|\nabla U_{x_n}(\xi)|^2 + \Lambda\ind_{\{|U_{x_n}(\xi)|>0\}}\big)\,d\xi \\
&=\int_{B_{rr_n}}\big(A_{x_n}\nabla U\cdot\nabla U + \Lambda\ind_{\{|U|>0\}}\big)\circ F_{x_n}(\xi)\,d\xi \\
&\leq \int_{B_{rr_n}}\big(A_{x_0}\nabla U\cdot\nabla U + \Lambda\ind_{\{|U|>0\}}\big)\circ F_{x_n}(\xi)\,d\xi  + \omega_d(rr_n)^dL^2c_{\text{\tiny\sc A}}|x_0-x_n|^{\delta_{\text{\tiny\sc A}}},
\end{align*}
where in the last inequality we have used that the coefficients $a_{ij}$ are $\delta_{\text{\tiny\sc A}}$-H\"{o}lder continuous, that is $\|A_{x_0}-A_{x_n}\|\leq c_{\text{\tiny\sc A}}|x_0-x_n|^{\delta_{\text{\tiny\sc A}}}$. We perform the change of variables $\tilde{\xi}=r_n^{-1}F^{-1}_{x_0}\circ F_{x_n}(\xi)$ and set $B=y_n+A^{-\sfrac12}_{x_0}A^{\sfrac12}_{x_n}\big[B_r\big]$ to get
\begin{equation}\label{e:comp_weiss1}
\frac{1}{(rr_n)^d}J(U_{x_n},rr_n) \leq \frac{1}{r^d}\int_B\big(|\nabla\tilde{B}_n|^2+\Lambda\ind_{\{|\tilde{B}_n|>0\}}\big)\,\big|\det(A^{-\sfrac12}_{x_n}A^{\sfrac12}_{x_0})\big|\,d\tilde{\xi} + \omega_dL^2c_{\text{\tiny\sc A}}|x_0-x_n|^{\delta_{\text{\tiny\sc A}}}.
\end{equation}
Moreover, since the coefficients of $A^{\sfrac12}$ are $\frac{\delta_{\text{\tiny\sc A}}}{2}$-H\"{o}lder continuous, we have the estimate of the determinant  $|\det(A^{-\sfrac12}_{x_n}A^{\sfrac12}_{x_0})\big| \leq 1+c_{\text{\tiny\sc A}}|x_0-x_n|^{\delta_{\text{\tiny\sc A}}/2}$ and the following estimate of the symmetric difference 
\begin{align*}
|B\triangle B_r(y_n)| &= |A^{-\sfrac12}_{x_0}A^{\sfrac12}_{x_n}\big[B_r\big]\triangle B_r|\leq \omega_dr^d\Big[\big(1+c_{\text{\tiny\sc A}}|x_0-x_n|^{\delta_{\text{\tiny\sc A}}/2}\big)^d-\big(1-c_{\text{\tiny\sc A}}|x_0-x_n|^{\delta_{\text{\tiny\sc A}}/2}\big)^d\Big]\\
&\leq \omega_dr^d\Big[\big(1+d2^dc_{\text{\tiny\sc A}}^d|x_0-x_n|^{\delta_{\text{\tiny\sc A}}/2}\big)-\big(1-d2^dc_{\text{\tiny\sc A}}^d|x_0-x_n|^{\delta_{\text{\tiny\sc A}}/2}\big)\Big]\leq r^dC|x_0-x_n|^{\delta_{\text{\tiny\sc A}}/2}.
\end{align*}
Similarly, for $n$ big enough so that $|y_0-y_n|\leq r/2$, we have
\begin{equation*}
|B_r(y_0)\triangle B_r(y_n)| \leq \omega_dr^d\bigg[\bigg(1+\frac{|y_0-y_n|}{r}\bigg)^d - \bigg(1-\frac{|y_0-y_n|}{r}\bigg)^d\bigg]\leq r^dC\frac{|y_0-y_n|}{r}.
\end{equation*}
Combining all these estimates \eqref{e:comp_weiss1} now gives (because $\tilde{B}_n$ is $\lambda_{\text{\tiny\sc A}}L$-Lipschitz continuous)
\begin{align}\label{e:comp_weiss2}
\nonumber\frac{1}{(rr_n)^d}J(U_{x_n},rr_n) &\leq \frac{1}{r^d}J(\tilde{B}_n,y_0,r)+\omega_dL^2c_{\text{\tiny\sc A}}|x_0-x_n|^{\delta_{\text{\tiny\sc A}}}+\frac{\lambda_{\text{\tiny\sc A}}^2L^2+\Lambda}{r^d}|B|c_{\text{\tiny\sc A}}|x_0-x_n|^{\delta_{\text{\tiny\sc A}}/2} + \\ 
&\quad +\frac{\lambda_{\text{\tiny\sc A}}^2L^2+\Lambda}{r^d}\Big[|B\triangle B_r(y_n)|+|B_r(y_0)\triangle B_r(y_n)|+|B|c_{\text{\tiny\sc A}}|x_0-x_n|^{\delta_{\text{\tiny\sc A}}/2} \Big]\\
\nonumber&\leq \frac{1}{r^d}J(\tilde{B}_n,y_0,r) + C|x_0-x_n|^{\delta_{\text{\tiny\sc A}}/2}+C\frac{|y_0-y_n|}{r} .
\end{align}
We now compare the boundary integral terms. Since $U_{x_n}(\xi)=r_n\tilde{B}_n(y_n+r_n^{-1}A^{-\sfrac12}_{x_0}A^{\sfrac12}_{x_n}(\xi))$ and by the change of variables $\tilde{\xi}=r_n^{-1}\xi+y_0$ we have
\begin{align*}
\int_{\partial B_{rr_n}}|U_{x_n}|^2(\xi)\,d\mathcal{H}^{d-1}(\xi) &= \int_{\partial B_{rr_n}}r_n^2|\tilde{B}_n|^2(y_n+r_n^{-1}A^{-\sfrac12}_{x_0}A^{\sfrac12}_{x_n}(\xi))\,d\mathcal{H}^{d-1}(\xi) \\
&= r_n^{d+1}\int_{\partial B_{r}(y_0)}|\tilde{B}_n|^2(y_n+A^{-\sfrac12}_{x_0}A^{\sfrac12}_{x_n}(\tilde{\xi}-y_0))\,d\mathcal{H}^{d-1}(\tilde{\xi}).
\end{align*}
Therefore, using that $\tilde{B}_n$ is $\lambda_{\text{\tiny\sc A}} L$-Lipschitz continuous, $\tilde{B}_n(y_n)=0$ and that $|y_0-y_n|\leq r/2$, we get that
\begin{align*}
&\frac{1}{r^{d+1}}\int_{\partial B_{r}(y_0)}|\tilde{B}_n|^2(\xi)\,d\mathcal{H}^{d-1}(\xi)-\frac{1}{(rr_n)^{d+1}}\int_{\partial B_{rr_n}}|U_{x_n}|^2(\xi)\,d\mathcal{H}^{d-1}(\xi)= \\
&\qquad\qquad=\frac{1}{r^{d+1}}\int_{\partial B_r(y_0)}\Big(|\tilde{B}_n|^2(\xi)-|\tilde{B}_n|^2(y_n+A^{-\sfrac12}_{x_0}A^{\sfrac12}_{x_n}(\xi-y_0)) \Big)\,d\mathcal{H}^{d-1}(\xi) \\
&\qquad\qquad\leq\frac{\lambda_{\text{\tiny\sc A}}^2L^2}{r^{d+1}}\int_{\partial B_r(y_0)}\big|A^{-\sfrac12}_{x_0}(A^{\sfrac12}_{x_0}-A^{\sfrac12}_{x_n})(\xi-y_0)+y_0-y_n\big|\,\big(|\xi-y_n| + \lambda_{\text{\tiny\sc A}}^2r \big)\,d\mathcal{H}^{d-1}(\xi) \\
&\qquad\qquad\leq C|x_0-x_n|^{\delta_{\text{\tiny\sc A}}/2}+C\frac{|y_0-y_n|}{r}.
\end{align*}
Now, in view of \eqref{e:comp_weiss2} we get \eqref{e:comp_weissa}. This completes the proof.
\end{proof}

In the following Lemma we prove that if $\tilde{B_n}$ is a blow-up sequence with fixed center converging to $\tilde{B}_0$, then locally the singular set of $\{|\tilde{B}_n|>0\}$ must lie close to the singular set of $\{|\tilde{B}_0|>0\}$ (see \cite[Lemma 4.2]{weiss-99} and \cite[Lemma 5.20]{mazzoleni-terracini-velichkov-17}).

\begin{lm}\label{l:conv_sing_sets}
Let $x_0\in\partial\O^\ast_1\cap D$ and let $B_n=B_{x_0,r_n}$ be a blow-up sequence converging in the sense of Proposition \ref{p:convblowup} to some $B_0\in\mathcal{BU}_U(x_0)$. We set $\tilde{B}_n=B_n\circ A^{\sfrac12}_{x_0}$, $\tilde{B}_0=B_0\circ A^{\sfrac12}_{x_0}$, $\tilde{\O}_n=\{|\tilde{B}_n|>0\}$ and  $\tilde{\O}_0=\{|\tilde{B}_0|>0\}$. Then, for every compact set $K\subset\R^d$ and every open set $\mathcal{O}\subset\R^d$ such that $\text{Sing}(\partial\tilde{\O}_0)\cap K\subset \mathcal{O}$, we have $\text{Sing}(\partial\tilde{\O}_n)\cap K\subset \mathcal{O}$ for $n$ large enough.
\end{lm}

\begin{proof}
Arguing by contradiction there exist a compact set $K\subset\R^d$ and an open set $\mathcal{O}\subset\R^d$ such that $\text{Sing}(\partial\tilde{\O}_0)\cap K\subset \mathcal{O}$ and a sequence $(y_n)_{n\in\N}\subset\text{Sing}(\partial\tilde{\O}_n)\cap K\setminus \mathcal{O}$. Up to a subsequence, $y_n$ converges to some $y_0\in K\setminus \mathcal{O}$. Since $\partial\tilde{\O}_n$ locally Hausdorff converges to $\partial\tilde{\O}_0$ by Proposition \ref{p:convblowup}, it follows that $y_0\in\partial\tilde{\O}_0$ and, since $\text{Sing}(\partial\tilde{\O}_0)\cap K\subset \mathcal{O}$, we have that $y_0$ is a regular point of $\partial\tilde{\O}_0$, that is $y_0\in\text{Reg}(\partial\tilde{\O}_0)$. 
Since, moreover, $\tilde{B}_0$ is solution of the Alt-Caffarelli functional and is one-homogeneous, it follows that $\frac{1}{\Lambda\omega_d}\lim_{r\rightarrow 0^+}W(\tilde{B}_0,y_0,r) = \frac{1}{2}$ (see \cite[Lemma 5.4]{mazzoleni-terracini-velichkov-17}). We now fix $r>0$ such that
\begin{equation*}
\frac{1}{\Lambda\omega_d}W(\tilde{B}_0,y_0,r)\leq \frac12 + \frac{\delta}{4},
\end{equation*}
where $\delta$ is the constant from Lemma \ref{l:densityUx}. Now, since $\lim_{n\rightarrow\infty}W(\tilde{B}_n,y_0,r)=W(\tilde{B}_0,y_0,r)$, it follows that for every $n$ large enough we have
\begin{equation}\label{e:conv_sing_sets1}
\frac{1}{\Lambda\omega_d}W(\tilde{B}_n,y_0,r)\leq \frac{1}{\Lambda\omega_d}W(\tilde{B}_0,y_0,r) +\frac{\delta}{4} \leq \frac12 + \frac{\delta}{3}.
\end{equation}
Set $x_n=x_0+r_nA^{\sfrac12}_{x_0}(y_n)\in\partial\O^\ast_1\cap D$ and notice that $x_n$ converges to $x_0$. By Lemma \ref{l:comp_weiss} and \eqref{e:conv_sing_sets1} we get that for every $n$ large enough
\begin{equation*}\label{e:conv_sing_sets2}
\frac{1}{\Lambda\omega_d}W(U_{x_n},rr_n)\leq \frac12 + \frac{\delta}{3} + C|x_0-x_n|^{\delta_{\text{\tiny\sc A}}/2}+C\frac{|y_0-y_n|}{r}.
\end{equation*}
On the other hand, by Proposition \ref{p:weissder}, the function $\varphi_n(s)=W(U_{x_n},s)+Cs^{\delta_{\text{\tiny\sc A}}}$ is non-decreasing and hence
\begin{align*}
\Theta_{U_{x_n}}(0)&=\frac{1}{\Lambda\omega_d}\lim_{s\rightarrow 0^+}W(U_{x_n},s)=\frac{1}{\Lambda\omega_d}\lim_{s\rightarrow 0^+}\varphi_n(s) \leq\frac{1}{\Lambda\omega_d}\varphi_n(rr_n) = \frac{1}{\Lambda\omega_d}W(U_{x_n},rr_n)+C(rr_n)^{\delta_{\text{\tiny\sc A}}} \\
&\leq \frac12 + \frac{\delta}{3} + C|x_0-x_n|^{\delta_{\text{\tiny\sc A}}/2}+C\frac{|y_0-y_n|}{r} + C(rr_n)^{\delta_{\text{\tiny\sc A}}}< \frac12 + \frac{\delta}{2},
\end{align*}
where the last inequality holds for $n$ large enough. It follows from Lemmas \ref{l:densityUx} and \ref{l:carac_reg_part} that $x_n$ is a regular point of $\O^\ast_1$, in contradiction with the fact that $y_n=A^{-\sfrac12}_{x_0}\Big[\frac{x_n-x_0}{r_n}\Big]\in\text{Sing}(\partial\tilde{\O}_n)$.
\end{proof}

We are now in position to prove Proposition \ref{p:dim_sing}.

\begin{proof}[Proof of Proposition \ref{p:dim_sing}]
{\it (1)} Let $x_0\in\partial\O^\ast_1\cap D$ and $B_0\in\mathcal{BU}_U(x_0)$ and set $\tilde{B}_0=B_0\circ A^{\sfrac12}_{x_0}$ and $\tilde{\O}_0=\{|\tilde{B}_0|>0\}$. By Lemma \ref{l:minscalAC}, $|\tilde{B}_0|$ is a local minimizer of the scalar Alt-Caffarelli functional and since $d<d^\ast$, it follows that $\partial\tilde{\O}_0$ is the graph of a $C^{1,\alpha}$ function near $0$ (see \cite[Section 3]{weiss-99}). In particular, the density of $\tilde{\O}_0$ at $0$ is $1/2$ and hence $\Theta_{U_{x_0}}(0)=1/2$ by \eqref{e:densityUx1}. In view of Lemma \ref{l:carac_reg_part} we get that $x_0\in\text{Reg}(\partial\O^\ast_1\cap D)$.

{\it (2)} Assume by contradiction that there exists a sequence $(x_n)_{n\in\N}\subset\text{Sing}(\partial\O^\ast_1\cap D)$ converging to some $x_0\in\partial\O^\ast_1\cap D$. Set $r_n=|x_0-x_n|$ and let $B_n:=B_{x_0,r_n}$ be a blow-up sequence converging (in the sense of Proposition \ref{p:convblowup}) to some blow-up limit $B_0\in\mathcal{BU}_U(x_0)$. 
We consider two cases:

Case 1: $\text{Sing}(\partial\tilde{\O}_0)\backslash\{0\}\neq\emptyset$. By a rotation we may assume that $e_d\in\R^d$ is a singular point of $\partial\O_0$. Notice that $u_0=|\tilde{B}_0|$ is solution of the scalar Alt-Caffarelli functional and is one-homogeneous. Consider a blow-up limit $u_{00}$ of $u_0$ at $e_d$. By \cite[Lemma 3.1]{weiss-99}, $\{u_{00}>0\}$ is a minimal cone with vertex $0$ such that the whole line $te_d$, $t\in\R$, consists of singular points. Then, by \cite[Lemma 3.2]{weiss-99}, denoting the restriction $\overline{u}=u_{00|\R^{d-1}}$, we have that $\{\overline{u}>0\}$ is a minimal cone of dimension $(d-1)$ which is singular at $0$. Now, either $0$ is the only singular point and we have a contradiction with the definition of $d^\ast$, or we can repeat this procedure and get a contradiction since there are no three-dimensional singular minimal cones.

Case 2: $\text{Sing}(\partial\tilde{\O}_0)\backslash\{0\}=\emptyset$. Let $r>0$ to be chosen later. By Lemma \ref{l:comp_weiss}, we have for every $n$ large enough
\begin{equation*}
W(U_{x_n},rr_n)\leq W(\tilde{B}_n,y_0,r) + C\frac{|y_0-y_n|}{r}+C|x_0-x_n|^{\delta_{\text{\tiny\sc A}}/2}.
\end{equation*}
Now, by Proposition \ref{p:weissder}, the function $\varphi_n(s)=W(U_{x_n},s)+Cs^{\delta_{\text{\tiny\sc A}}}$ is non-decreasing and, since $x_n\in\text{Sing}(\partial\O^\ast_1\cap D)$, by Lemmas \ref{l:carac_reg_part} and \ref{l:densityUx} we have that $\frac{1}{\Lambda\omega_d}\lim_{s\rightarrow 0^+}W(U_{x_n},s)\geq \frac{1}{2}+\delta$. Therefore, we have
\begin{align}\label{e:dim_sing3}
\nonumber\frac{1}{2}+\delta &\leq \frac{1}{\Lambda\omega_d}\lim_{s\rightarrow 0^+}W(U_{x_n},s) = \frac{1}{\Lambda\omega_d}\lim_{s\rightarrow 0^+}\varphi_n(s) \leq \varphi_n(rr_n)= \frac{1}{\Lambda\omega_d}W(U_{x_n},rr_n)+C(rr_n)^{\delta_{\text{\tiny\sc A}}} \\
&\leq \frac{1}{\Lambda\omega_d}W(\tilde{B}_n,y_0,r) + C\frac{|y_0-y_n|}{r}+C|x_0-x_n|^{\delta_{\text{\tiny\sc A}}/2} +C(rr_n)^{\delta_{\text{\tiny\sc A}}}.
\end{align}
Now, since $y_0\in\partial\tilde{\O}_0\backslash\{0\}$ is a regular point of $\partial\tilde{\O}_0$ (and also because $\tilde{B}_0$ is solution of the Alt-Caffarelli functional and is one-homogeneous), it follows that $\frac{1}{\Lambda\omega_d}\lim_{r\rightarrow 0^+}W(\tilde{B}_0,y_0,r)=\frac{1}{2}$ (see \cite[Lemma 5.4]{mazzoleni-terracini-velichkov-17}). Using also that $\lim_{n\rightarrow\infty}W(\tilde{B}_n,y_0,r) = W(\tilde{B}_0,y_0,r)$, it follows that we can choose $r>0$ small enough such that for every $n$ large enough we have
\begin{equation*}
\frac{1}{\Lambda\omega_d}W(\tilde{B}_n,y_0,r) \leq \frac{1}{\Lambda\omega_d}W(\tilde{B}_0,y_0,r) + \frac{\delta}{4} \leq \frac{1}{2} + \frac{\delta}{2}.
\end{equation*}
Therefore, passing to the limit $n\rightarrow\infty$ in the equation \eqref{e:dim_sing3} gives a contradiction.

{\it (3)} Assume by contradiction that $\mathcal{H}^{d-d^\ast+s}(\text{Sing}(\partial\O^\ast_1\cap D))>0$ for some $s>0$. By Lemma \ref{l:conv_sing_sets} and \cite[Lemmas 4.3 and 4.4]{weiss-99} there exists $x_0\in\partial\O^\ast_1\cap D$ and a blow-up limit $B_0\in\mathcal{BU}_U(x_0)$ such that $\mathcal{H}^{d-d^\ast+s}(\text{Sing}(\partial\tilde{\O}_0))>0$, where we have set $\tilde{B}_0=B_0\circ A^{\sfrac12}_{x_0}$ and $\tilde{\O}_0=\{|\tilde{B}_0|>0\}$. Since $|\tilde{B}_0|$ is a minimizer of the Alt-Caffarelli functional and is one-homogeneous, the dimension reduction procedure in \cite[Lemma 4.5]{weiss-99} applies and yields to a minimizer $u:\R^{d^\ast}\rightarrow\R$ of the Alt-Caffarelli functional such that $\mathcal{H}^{s}(\text{Sing}(\partial\{u>0\}))>0$, in contradiction with \cite[Lemma 4.1]{weiss-99}.
\end{proof}

\bigskip\bigskip
\noindent {\bf Acknowledgments.} 
This work was partially supported  by the French Agence Nationale de la Recherche (ANR) with the projects GeoSpec (LabEx PERSYVAL-Lab, ANR-11-LABX-0025-01) and the project SHAPO (ANR-18-CE40-0013). 

\bibliographystyle{plain}
\bibliography{bib-t}

\end{document}